\documentclass[english]{amsart}

\usepackage{esint}
\usepackage[svgnames]{xcolor} 
\usepackage[colorlinks,citecolor=red,pagebackref,hypertexnames=false,breaklinks]{hyperref}
\usepackage{pgf,tikz}
\usepackage{pdfsync}

\usepackage{dsfont}
\usepackage{url}
\usepackage[utf8]{inputenc}
\usepackage[T1]{fontenc}
\usepackage{lmodern}
\usepackage{babel}
\usepackage{mathtools}  
\usepackage{amssymb}
\usepackage{lipsum}
\usepackage{mathrsfs}
\usepackage{color}
\usepackage{stmaryrd}
\usepackage{soul}

\newtheorem{theorem}{Theorem}[section]
\newtheorem{proposition}{Proposition}[section]
\newtheorem{lemma}{Lemma}[section]

\newtheorem{corollary}{Corollary}[section]

\numberwithin{equation}{section}

\title[Quantitative Borg-Levinson theorem]{Quantitative Borg-Levinson theorem for the magnetic Schr\"odinger operator with unbounded electrical potential }

\author{Mourad Choulli}
\address{Universit\'{e} de Lorraine, 34 cours L\'{e}opold, 54052 Nancy cedex, France}
\email{mourad.choulli@univ-lorraine.fr}

\author{Hiroshi Takase}
\address{Department of Mathematics, Okayama University, 3-1-1 Tsushima-naka, Kita-ku, Okayama 700-8530, Japan}
\email{takase@math.okayama-u.ac.jp}

\date{}

\begin{document}
\begin{abstract}
In \cite{Ch1}, the first author established a quantitative Borg-Levinson theorem for the Schr\"odinger operator with an unbounded potential. In the present work, we extend the results of \cite{Ch1} to the magnetic Schrödinger operator. We address both the isotropic and anisotropic cases. We establish H\"older-type stability inequalities for determining the electric potential or the magnetic field from the corresponding boundary spectral data.
\end{abstract}

\subjclass[2010]{35R30, 35R01, 35P99, 35J10, 35J15}

\keywords{Magnetic Schr\"odinger operator, quantitative Borg-Levinson theorem, electrical potential, magnetic field, boundary spectral data.}

\thanks{This work was supported by JSPS KAKENHI Grant Numbers JP25K17280, JP23KK0049.}

\maketitle

\tableofcontents

\section*{General introduction}

Since the seminal article by Nachman, Sylvester and Uhlmann \cite{NSU}, significant progress has been made in multidimensional inverse spectral theory. More specifically, this has been achieved both by reducing the known part of the boundary spectral data (BSD) and by reducing the regularity of the potential. By BSD, we mean the eigenvalues of the Schr\"odinger operator under Dirichlet boundary conditions and the normal derivatives of the corresponding eigenfunctions. Roughly speaking, the main idea in the work by Nachman, Sylvester and Uhlmann \cite{NSU} was to show that the derivatives, of an order dependent of the dimension of the ambient space, of the kernels of an analytic family of Dirichlet-to-Neumann (DtN) maps are entirely determined by the BSD. The uniqueness of determining the bounded potential is then obtained from known results concerning the determination of the bounded potential based on knowledge of the corresponding DtN map. The quantitative version of the result presented in \cite{NSU} was obtained by Alessandrini and Sylvester \cite{AS} by relating the BSD to an hyperbolic DtN map. In \cite{Is}, Isozaki employs a method borrowed from the Born-approximation to demonstrate that the potential can be uniquely determined even if a finite part of the BSD is unknown. The first author and Stefanov \cite{ChS} established that if BSD corresponding to two potentials cannot be sufficiently close one to each other, in some appropriate sense, unless they are the same. A quantified version of this result was also given in \cite{ChS}. The method introduced by Isozaki \cite{Is} was modified by Kavian, Kian and Soccorsi \cite{KKS} to prove that, in the vicinity of a given potential, knowledge of the asymptotic part of the BSD is sufficient to  uniquely determine the bounded potential. We note that the results in \cite{KKS}  initially concern an inverse spectral problem in a periodic waveguide. However, this results can be readily adapted to the standard Schr\"odinger equation. For further details, the reader is referred to \cite{So}. 

The uniqueness in the case of potentials in $L^s$, $s>\frac{n}{2}$, is due to P\"aiv\"arinta and Serov \cite{PS} and extended to the optimal case $s=\frac{n}{2}$ by Pohjola \cite{Po}. The quantitative version of Pohjola's result was recently obtained by the first author in \cite{Ch1} with an extension to the anisotropic case. We emphasize here that most Hilbertian analysis used for bounded potentials are no longer valid for unbounded potentials. We overcome this difficulty by using finer resolvent estimates. The case of the Schrödinger equation with Robin boundary conditions was recently studied by the first author, Metidji and Soccorsi \cite{CMS1,CMS2}. In these two papers, various uniqueness and stability inequalities were established.

Following the method described in \cite{KKS}, Kian obtained in \cite{Ki}  the first results  enabling the determination of both the electric potential and the magnetic field, appearing in a magnetic Schr\"odinger equation, from the corresponding asymptotic BSD. Later, these results were generalized to the anisotropic case in \cite{BCDKS}. In this later work, the case of the magnetic Schr\"odinger equation under Neumann boundary condition was also considered. In both \cite{Ki} and \cite{BCDKS} the electric potential is assumed to be bounded. The stability result by Alessandrini and Sylvester was reformulated in a suitable topology in \cite{Ch}. The method employed in \cite{Ch} was more recently adapted by Liu, Quan, Saksala and Yan \cite{LQSY} to establish a H\"older stability inequality for both electric potential and magnetic field from the corresponding BSD of the Schr\"odinger equation in a simple Riemannian manifold. It should be noted that there is an obstacle to the unique determination of the magnetic potential, due to the gauge invariance of the corresponding BSD (see, for example, \cite{Ki} for further details).

The determination of the Riemannian metric from the corresponding BSD by the so-called boundary control method was initiated by Belishev \cite{Bel87} (see also \cite{Bel92,KKL}), where uniqueness result (up to gauge equivalence) has been established. The case where the normal derivative of the eigenfunctions is only known on part of the boundary was obtained by Katchalov and Kurylev \cite{KK}. Various quantitative results concerning the inverse spectral problems mentioned above was proven by the first author and Yamamoto \cite{CY}. The results of \cite{CY} were recently improved by the first author in \cite{Ch2}.  

In the present work, we extend the results of \cite{Ch1} to the magnetic case, that is, when the potential belongs to the optimal class $L^{\frac{n}{2}}$. Note that this class of potentials is the best possible for performing the direct spectral analysis of the Sch\"odinger operator in the Hilbertian framework. The direct spectral problem is still well defined in the presence of a magnetic potential in $W^{1,n}$. However, for the inverse spectral problem,  $W^{1,\infty}$ regularity of the magnetic potential is required in the isotropic case and $W^{2,\infty}$ regularity of the magnetic potential is required in the anisotropic case. We do not know whether the regularity of the magnetic potential can be improved. Furthermore, the case of partial or asymptotic BSD remains an open problem.

Without claiming to be exhaustive, we mention a few additional references \cite{BCY,BCY2,BKMS,GL,KOM,KS,Ma,No}.

For clarity, the remainder of this text is divided into two parts. The first deals with the isotropic case, while the second examines the anisotropic case. However, we do not repeat the results obtained in the first part, which require only slight modifications to be valid in the second. The H\"older stability inequalities we have established are similar to those obtained in the case of bounded electric potentials. Due to the dimensionality dependence of Sobolev's embedding  theorems, we have limited our study to dimensions greater than or equal to 5. Extending this to dimensions 2, 3, and 4 requires some modifications. A detailed study of this case will be the subject of future work.

\section*{Part I: Isotropic case}\label{part1}

\section{Introduction}

\subsection{Main notations} Let $\Omega$ be a Lipschitz bounded domain of $\mathbb{R}^n$, $n\ge 5$ whose boundary will be denoted by $\Gamma$, and $\Omega_0$ be an open bounded subset of $\mathbb{R}^n$ such that $\Omega_0\Supset \Omega$. The following notations will be used in all of this text:
\[
m=\frac{n}{2},\quad p=\frac{2n}{n+2},\quad q=\frac{2n}{n+4}.
\] 
The respective conjugate of $p$ and $q$ will be denoted $p'$ and $q'$. That is,
\[
p'=\frac{2n}{n-2},\quad q'=\frac{2n}{n-4}.
\]

For $1\le r\le \infty$, $s\ge 0$ and $k\ge 1$ an integer, we set
\[
W_\ast^{s,r}(\mathbb{R}^n,\mathbb{K}^k):=\{f\in W^{s,r}(\mathbb{R}^n,\mathbb{K}^k),\; \mathrm{supp}(f)\subset \Omega_0\},\quad \mathbb{K}\in \{\mathbb{R},\mathbb{C}\},
\]
which we endow with the norm of $W^{s,r}(\mathbb{R}^n,\mathbb{K}^k)$. For simplicty, we set $L_\ast^r(\mathbb{R}^n,\mathbb{K}^k):=W_\ast^{0,r}(\mathbb{R}^n,\mathbb{K}^k)$. 

It the following, for all  $1\le r\le \infty$ and $s\ge 0$, the norm of $W^{s,r}(X)$, $X\in \{\mathbb{R}^n,\Omega,\Gamma\}$, will denoted respectively by $\|\cdot\|_{s,r} $, with the usual convention that $W^{0,r}(X)=L^r(X)$ and identifying $W^{s,2}(X)$ with $H^s(X)$.
If there is no confusion, we will also use the notation $\|\cdot \|_{s,r}$ for norm of $W^{s,r}(X,\mathbb{K}^n)$, $\mathbb{K}\in \{\mathbb{R},\mathbb{C}\}$, which is given as follows:
\begin{align*}
&\|f\|_{s,r}:=\left(\sum_{j=1}^n\|f^j\|_{s,r}^r\right)^{\frac{1}{r}},\quad f=(f^1,\ldots,f^n)\in W^{s,r}(X,\mathbb{K}^n)\quad \mathrm{if}\; r<\infty,
\\
&\|f\|_{s,\infty}:=\sum_{j=1}^n\|f^j\|_{s,\infty},\quad f=(f^1,\ldots,f^n)\in W^{s,\infty}(X,\mathbb{K}^n)\quad \mathrm{if}\; r=\infty.
\end{align*}

Let $1\le r<r_1$ and $r_2=\frac{rr_1}{r_1-r}$ (note that $\frac{1}{r_2}+\frac{1}{r_1}=\frac{1}{r}$). The following generic formula will be useful in the sequel. It is a direct consequence of H\"older's inequality
\begin{equation}\label{41}
\|uv\|_{0,r}\le \|u\|_{0,r_1}\|v\|_{0,r_2},\quad u\in L^{r_1}(\Omega),\; v\in L^{r_2}(\Omega).
\end{equation}

According to Poincar\'e's inequality, we have 
\[
\varkappa:=\sup\{\|w\|_{0,2};\; w\in H_0^1(\Omega),\; \|\nabla w\|_{0,2}=1\}<\infty,
\]
and, since $H_0^1(\Omega)$ is continuously embedded in $L^{p'}(\Omega)$, we have also
\[
\kappa:=\sup \{ \|w\|_{0,p'};\; w\in H_0^1(\Omega),\; \|\nabla w\|_{0,2}=1\}<\infty.
\]

Let $a\in L^n(\mathbb{R}^n,\mathbb{R}^n)$ and $w\in H_0^1(\Omega)$. It follows from \eqref{41}
\[
\|wa\|_{0,2}\le \|a\|_{0,n}\|w\|_{0,p'}.
\]
Whence,
\begin{equation}\label{7}
\|wa\|_{0,2}\le \kappa \|a\|_{0,n}\|\nabla w\|_{0,2}.
\end{equation}

For $a\in L^n(\mathbb{R}^n,\mathbb{R}^n)$, the following notation will be used hereinafter:
\[
\nabla_a:=\nabla +ia .
\]

Fix $0<\mathfrak{c}<\kappa^{-1}$ and set
\[
\mathcal{A}:=\{a\in L_\ast^n(\mathbb{R}^n,\mathbb{R}^n);\; \|a_{|\Omega}\|_{0,n}\le \mathfrak{c}\}.
\]
Using \eqref{7}, we verify that
\begin{equation}\label{1}
\mathfrak{c}_-\|\nabla w\|_{0,2}\le \|\nabla_a w\|_{0,2}\le \mathfrak{c}_+\|\nabla w\|_{0,2},\quad a\in \mathcal {A},\; w\in H_0^1(\Omega),
\end{equation}
where $\mathfrak{c}_\pm=1\pm \mathfrak{c}\kappa$.

Note that, without the condition $\|a_{|\Omega}\|_{0,n}\le \mathfrak{c}$, $\|\nabla_a \cdot\|_{0,2}$ remains a Hilbertian norm on $H_0^1(\Omega)$. However, it is not necessarily equivalent to the norm $\|\nabla \cdot\|_{0,2}$. For further details, we refer to \cite[Section 2]{EL}.

In this text, we will use the following now-standard abbreviations: BSD for ``boundary spectral data'' and DtN for ``Dirichlet-to-Neumann''.

\subsection{Statement of the main results}

 In this subsection, we assume that $\Omega$ is of class $C^{1,1}$. Fix $V_0\in L_\ast^m(\mathbb{R}^n,\mathbb{R})$ non negative and non identically equal to zero, and let
\[
\mathcal{V}:=\{ V\in L_\ast^m(\mathbb{R}^n,\mathbb{R});\; |V|\le V_0\}.
\]

For $b=(a,V)\in \mathcal{A}\times \mathcal{V}$, define the sesquilinear form 
\begin{equation}\label{ses}
S^b(u,v)=\int_\Omega \left[ \nabla_au\cdot \overline{\nabla_av}-Vu\overline{v}\right]dx,\quad u,v\in H_0^1(\Omega).
\end{equation}
Then define the  bounded operator $A^b:H_0^1(\Omega)\rightarrow H^{-1}(\Omega)$ as follows
\begin{equation}\label{op}
\langle A^bu,\overline{v}\rangle=S^b(u,v),\quad u,v\in H_0^1(\Omega).
\end{equation}
where $\langle \cdot ,\cdot \rangle$ is the duality pairing between $H_0^1(\Omega)$ and its dual $H^{-1}(\Omega)$.

In the following section, we will establish that the spectrum of $A^b$, denoted by $\sigma(A^b)$, reduces to a sequence of eigenvalues $(\lambda_k^b)$ satisfying

\[
-\infty <\lambda_1^b\le \lambda_2^b\le \ldots \le \lambda_k^b\le \ldots \quad \text{and}\quad \lim_{k\rightarrow \infty}\lambda_k^b=\infty.
\]
Furthermore, $L^2(\Omega)$ admits an orthonormal basis consisting of a sequence of eigenfunctions $(\phi_k^b)$ such that $\phi_k^b\in H_0^1(\Omega)$ and 
\begin{equation}\label{9}
S^b(\phi_k^b,v)=\lambda_k^b (\phi_k^b|v),\quad v\in H_0^1(\Omega),\; k\ge 1.
\end{equation}
Here and henceforth, $(\cdot|\cdot)$ denotes the usual inner product of $L^2(\Omega)$. 

As  $W_\ast^{1,m}(\mathbb{R}^n,\mathbb{R}^n)$ is continuously embedded in $L_\ast^n(\mathbb{R}^n,\mathbb{R}^n)$, we can define
\[
\mathcal{A}_+:=\{a\in W_\ast ^{1,m}(\mathbb{R}^n,\mathbb{R}^n);\; \|a_{|\Omega}\|_{0,n}\le \mathfrak{c},\; \|a\|_{1,m}\le \tilde{\mathfrak{c}}\},
\]
where $\tilde{\mathfrak{c}}>0$ is a constant. We demonstrate in the next section that if $b=(a,V)\in \mathcal{A}_+\times \mathcal{V}$, then  $\phi_k^b\in W^{2,q}(\Omega)$ and therefore $\psi_k^b:=\partial_{\nu_a}\phi_k^b:=\nabla_a \phi_k^b\cdot \nu\in L^q(\Gamma)$, for all $k\ge 1$. We keep these notations in the rest of this text.

Before stating our first main result, we need new definitions. For $b_j=(a_j,V_j)\in \mathcal{A}_+\times \mathcal{V}$, $j=1,2$, in the rest of this text we use the notations
\[
\delta (b_1,b_2):=\sum_{k\ge 1}k^{-\frac{2\ell }{n}}\left[|\lambda_k^{b_1}-\lambda_k^{b_2}|+\|\psi_k^{b_1}-\psi_k^{b_2}\|_{0,q}\right].
\]
where $\ell$ is the smallest integer satisfying $\ell >\frac{n}{2}+1$. Note that, according to \eqref{8} and \eqref{11} below, we have  $\delta (b_1,b_2)<\infty$.

In rest of this text, we fix $s\in ]1+\frac{1}{q},2[$ and $2<r<p'$ such that
\[
\upsilon:=\frac{3-s-s_1-\varsigma}{2}>0,
\]
where $s_1=\frac{3s-2}{s+2}$ and $\varsigma=\frac{n(r-2)}{2r}$. Note that we used the fact that $3-s-s_1=\frac{8-s^2-2s}{2+s}>0$. Let then $t=\frac{8n}{4n+1}$, $\mathfrak{a}:=\frac{2r}{r-2}$ and $\sigma:=\frac{3}{\upsilon}$, and define
\begin{align*}
&\mathcal{A}_0:=\{a\in W_\ast^{1,\infty}(\mathbb{R}^n,\mathbb{R}^n)\cap W_\ast^{2,t}(\mathbb{R}^n,\mathbb{R}^n);\; \|a_{|\Omega}\|_{0,n}\le \mathfrak{c}\},
\\
&\mathcal{V}_0:=\{(V_1,V_2)\in \mathcal{V}\times \mathcal{V};\; |V_1-V_2|\le W_0\},
\end{align*}
where $W_0\in L_\ast^\mathfrak{a}(\mathbb{R}^n,\mathbb{R})$ is nonnegative and non identically equal to $0$.

We use the notation $\|\cdot \|_{-1,2}$ to denote the norm on $H^{-1}(\mathbb{R}^n)$ as well as on $H^{-1}(\Omega)$.

\begin{theorem}\label{thm1}
Let $a\in \mathcal{A}_0$, $(V_1,V_2)\in \mathcal{V}_0$ and $b_j=(a,V_j)$, $j=1,2$. Then we have
\[
\|V_1-V_2\|_{-1,2}\le \mathbf{c}_\ast\delta(b_1,b_2)^{\beta_0},
\]
where
\[
\beta_0=\frac{1}{2(n+2)\left((\sigma+1)\ell+\frac{17}{4}\right)}.
\]
and $\mathbf{c}_\ast=\mathbf{c}_\ast(n,\Omega,\Omega_0,a,V_0,W_0,\mathfrak{c},\tilde{\mathfrak{c}},s,r)>0$ is a constant.
\end{theorem}

Before stating our second main result, we introduce additional notations. Let $K\subset \Omega$ be compact, $\mathbf{a}\in L_\ast^\infty (\mathbb{R}^n,\mathbb{R})$ and $W_0\in L_\ast^\mathfrak{a}(\mathbb{R}^n,\mathbb{R})$ be nonnegative and not identically zero. Let us define $\mathscr{B}$ as the set of pairs
\[
(b_1,b_2)=((a_1,V_1),(a_2,V_2))\in [\mathcal{A}_+\times \mathcal{V}]^2
\]
satisfying $\mathrm{supp}(a_1-a_2)\subset K$ and
\[
|a_1-a_2|\le \mathbf{a},\quad |i\mathrm{div}(a_1-a_2)-|a_1|^2+|a_2|^2+ V_1-V_2|\le W_0.
\]

The subset of $\mathscr{B}$ consisting of those pairs $(b_1,b_2)=((a_1,V_1),(a_2,V_2))\in \mathscr{B}$ such that $a_1,a_2\in W_\ast^{1,\infty}(\mathbb{R}^n,\mathbb{R}^n)$, $a_1-a_2\in W_\ast^{2,\infty}(\mathbb{R}^n,\mathbb{R}^n)\cap W_\ast^{3,t}(\mathbb{R}^n,\mathbb{R}^n)$ and
\[
\|a_1\|_{1,\infty}\le \tilde{\mathfrak{c}},\quad \|a_2\|_{1,\infty}\le \tilde{\mathfrak{c}},\quad \|a_1-a_2\|_{2,\infty}+\|a_1-a_2\|_{3,t}\le \tilde{\mathfrak{c}}
\]
will be denoted by $\mathscr{B}_+$. In particular, we have $\tilde{a}=\chi_\Omega(a_1-a_2)\in H^1(\mathbb{R}^n,\mathbb{R}^n)$ and
\[
\|\tilde{a}\|_{1,2}\le c_\Omega \tilde{\mathfrak{c}},
\]
where $c_\Omega>0$ is a constant depending only on $\Omega$.

For $w=(w^1,\ldots ,w^n)\in H^1(\mathbb{R}^n,\mathbb{R}^n)$, the $2$-form $dw$ is given by
\[
dw=\sum_{k<\ell}(\partial_\ell w^k-\partial_kw^\ell)dx^\ell\wedge dx^k.
\]
We will use the notation
\[
\|dw\|_{0,2}=\left(\sum_{k< \ell}\|\partial_\ell w^k-\partial_kw^\ell\|_{0,2}^2\right)^{\frac{1}{2}}.
\]

\begin{theorem}\label{thm2}
Let $(b_1,b_2)=((a_1,V_1),(a_2,V_2))\in \mathscr{B}_+$. Then we have
\[
\|d\tilde{a}\|_{0,2}\le \mathbf{c}_+\delta(b_1,b_2)^{\beta_1},
\]
where $\tilde{a}=\chi_\Omega(a_1-a_2)$,
\[
\beta_1=\frac{\mathfrak{t}}{(\sigma +1)\ell+4+(n+1)\mathfrak{t}},
\]
with $\mathfrak{t}:=\min(\frac{3}{4(n+2)},\frac{1}{16})$, and $\mathbf{c}_+=\mathbf{c}_+(n,\Omega,\Omega_0,\mathbf{a},V_0,W_0,\mathfrak{c},\tilde{\mathfrak{c}},K,s,r)>0$ is a constant.
\end{theorem}

Note that $\mathfrak{t}=\frac{1}{16}$ for all $n\ge 10$. Furthermore, by examining the proof of Theorems \ref{thm1} and \ref{thm2}, one could verify that $t$ can be replaced by $t_\theta=\frac{2n}{n+2\theta}$, where $0<\theta<\frac{1}{2}$ is chosen arbitrarily.

We conclude this section with the following uniqueness of the determination of $b=(a,V)$ from the corresponding BSD.

\begin{corollary}\label{coroll}
Assume that $\mathbb{R}^n\setminus \Omega$ is connected. For $j=1,2$, let $b_j=(a_j,V_j)\in \mathcal{A}_0\times L_\ast^m(\mathbb{R}^n)$ be such that $a_1-a_2\in W_\ast^{2,\infty}(\mathbb{R}^n,\mathbb{R}^n)\cap W_\ast^{3,t}(\mathbb{R}^n,\mathbb{R}^n)$ with $\mathrm{supp}(a_1-a_2)\subset \Omega$, and $V_1-V_2\in L_\ast^\mathfrak{a}(\mathbb{R}^n)$.
If 
\[
\lambda_k^{b_1}=\lambda_k^{b_2}, \quad \psi_k^{b_1}=\phi_k^{b_2}, \quad k\ge1,
\]
then $da_1=da_2$ and $V_1=V_2$.
\end{corollary}

\begin{proof}
By first applying Theorem \ref{thm2}, we obtain $da_1=da_2$. We then slightly modify the proof \cite[Theorem 1.1]{Ki} to obtain from Theorem \ref{thm1} that $V_1=V_2$.
\end{proof}

Under additional conditions, we can reduce the BSD, as shown in the following result.

\begin{corollary}
Let $\Omega_1$ be an open connected neighborhood of $\Gamma$ in $\overline{\Omega}$ such that $\tilde{\Omega}=\Omega \setminus \overline{\Omega_1}\ne\emptyset$, and let $S$ be a nonempty open subset of $\Gamma$. In addition of the assumptions of Corollary \ref{coroll}, assume that $V_1{_{|\Omega_1}}=V_2{_{|\Omega_1}}:=V\in L^\mathfrak{a}(\Omega_1)$ and $a_1{_{|\Omega_1}}=a_2{_{|\Omega_1}}:=a$.
If 
\[
\lambda_k^{b_1}=\lambda_k^{b_2},\quad \psi_k^{b_1}{_{|S}}=\psi_k^{b_2}{_{|S}},\quad k\ge 1,
\]
then $da_1=da_2$ and $V_1=V_2$.
\end{corollary}

\begin{proof}
Let $k\ge 1$. We verify that
\[
\Delta_b(\phi_k^{b_1}-\phi_k^{b_2})=0\quad  \mathrm{in}\; \tilde{\Omega},
\]
where $b=(a,V)$. Since $\phi_k^{b_1}{_{|S}}=\phi_k^{b_2}{_{|S}}$ and $\psi_k^{b_1}{_{|S}}=\psi_k^{b_2}{_{|S}}$, it follows from \cite[Theorem 11.3]{CT} (unique continuation from the Cauchy data on $S$) that  $\phi_k^{b_1}{_{|\tilde{\Omega}}}=\phi_k^{b_2}{_{|\tilde{\Omega}}}$. In particular, we have $\psi_k^{b_1}=\psi_k^{b_2}$. We then apply Corollary \ref{coroll} to complete the proof.
\end{proof}

\section{Spectral analysis and resolvent estimates}

\subsection{Spectral analysis of the magnetic Schr\"odinger operator}

The following useful lemma corresponds to \cite[Lemma 2.1]{Ch1}.

\begin{lemma}\label{lem8}
For all $\epsilon >0$ there exists a constant $\mathbf{c}_\epsilon=\mathbf{c}_\epsilon (n,\Omega, V_0,\epsilon)>0$ such that
\[
\|V|u|^2\|_{0,1}\le\epsilon \|\nabla u\|_{0,2}+\mathbf{c}_\epsilon \|u\|_{0,2},\quad V\in \mathcal{V},\; u\in H^1(\Omega). 
\]
\end{lemma}

For $b=(a,V)\in \mathcal{A}\times \mathcal{V}$, let $S^b$ be the sesquilinear form given by \eqref{ses}. That is,
\[
S^b(u,v)=\int_\Omega \left[ \nabla_au\cdot \overline{\nabla_av}-Vu\overline{v}\right]dx,\quad u,v\in H_0^1(\Omega).
\]

Let $u,v\in H_0^1(\Omega)$. By applying twice \eqref{41}, we obtain
\begin{align*}
\|Vu\overline{v}\|_{0,1}&=\|[|V|^{\frac{1}{2}}u][|V|^{\frac{1}{2}}\overline{v}]\|_{0,1}
\\
&\le \||V|^{\frac{1}{2}}u\|_{0,2}\||V|^{\frac{1}{2}}\overline{v}]\|_{0,2}
\\
&\le \|V_0\|_{0,m}  \|u\|_{0,p'} \|v\|_{0,p'}.
\end{align*}
Therefore, we have
\begin{equation}\label{2}
\|Vu\overline{v}\|_{0,1}\le \kappa ^2 \|V_0\|_{0,m}  \|\nabla u\|_{0,2} \|\nabla v\|_{0,2}.
\end{equation}

On the other hand, in light of \eqref{1}, by applying Cauchy-Schwarz's inequality, we get
\begin{equation}\label{3}
\|\nabla_au\cdot \overline{\nabla_av}\|_{0,1}\le \mathfrak{c}_+^2\|\nabla u\|_{0,2}\|\nabla v\|_{0,2}.
\end{equation}
\eqref{2} and \eqref{3} show that $S^b$ is continuous on $H_0^1(\Omega)\times H_0^1(\Omega)$. Here and in what follows, we often endow $H_0^1(\Omega)$ with the norm $\|\nabla \cdot\|_{0,2}$ which is, of course, equivalent to the norm $\|\cdot \|_{1,2}$.

Next, we prove that $S^b$ is coercive. From Lemma \ref{lem8}, for all $\epsilon >0$, there exists a constant $\mathbf{c}_\epsilon=\mathbf{c}_\epsilon(n,\Omega, V_0,\epsilon)>0$ such that we have
\begin{equation}\label{4}
\|V|u|^2\|_{0,1}\le \epsilon \|\nabla u\|_{0,2}^2+\mathbf{c}_\epsilon\|u\|_{0,2}^2,\quad u\in H_0^1(\Omega).
\end{equation}

By combining \eqref{1} and \eqref{4}, we obtain
\[
S^b(u,u)\ge \mathfrak{c}_-^2\|\nabla u\|_{0,2}^2-\epsilon \|\nabla u\|_{0,2}^2-\mathbf{c}_\epsilon\|u\|_{0,2}^2,\quad u\in H_0^1(\Omega).
\]
Choosing $\epsilon=\frac{\mathfrak{c}_-^2}{2}$ in the inequality above gives
\begin{equation}\label{5}
S^b(u,u)\ge \frac{\mathfrak{c}_-^2}{2}\|\nabla u\|_{0,2}^2-\mathbf{c}\|u\|_{0,2}^2,\quad u\in H_0^1(\Omega).
\end{equation}
Here and henceforth, $\mathbf{c}=\mathbf{c}(n,\Omega, V_0,\mathfrak{c})>0$ denotes a generic constant. In other words, we established that $S^b$ is coercive.

Let $A^b$ be defined by \eqref{op}:  $A^b:H_0^1(\Omega)\rightarrow H^{-1}(\Omega)$ is given as follows
\[
\langle A^bu,\overline{v}\rangle=S^b(u,v),\quad u,v\in H_0^1(\Omega).
\]
The operator $A^b$ is clearly self-adjoint. By applying \cite[Theorem 2.37]{Mc}, we find that the spectrum of $A^b$, denoted by $\sigma(A^b)$, reduces to a sequence of eigenvalues $(\lambda_k^b)$ satisfying
\[
-\infty <\lambda_1^b\le \lambda_2^b\le \ldots \le \lambda_k^b\le \ldots ,\quad \lim_{k\rightarrow \infty}\lambda_k^b=\infty.
\]
Moreover, $L^2(\Omega)$ admits an orthonormal basis consisting of a sequence of eigenfunctions $(\phi_k^b)$ such that $\phi_k^b\in H_0^1(\Omega)$ and \eqref{9} holds.

It follows from \eqref{1} and \eqref{2} that
\begin{equation}\label{6}
S^b(u,u)\le (\mathfrak{c}_++\kappa^2\|V_0\|_{0,m})\|\nabla u\|_{0,2}^2, \quad u\in H_0^1(\Omega).
\end{equation}

In the following, $A^0$ will be $A^b$ when $b=0$. The corresponding sequence of eigenvalues will be denoted by $(\lambda_k^0)$. In view of \eqref{5} and \eqref{6}, we obtain from the min-max principle
\[
\varpi^{-1}\lambda_k^0-\varpi_0 \le \lambda_k^b\le \varpi \lambda_k^0,\quad k\ge 1.
\]
Here, $\varpi_0=\varpi_0(n,\Omega,V_0,\mathfrak{c})>0$ and $\varpi=\varpi(n,\Omega,V_0,\mathfrak{c})>1$ are constants. 

By replacing $\varpi_0$ and $\varpi$ with similar constants, the inequality above together with Weyl's asymptotic formula for $A^0$ yield
\begin{equation}\label{8}
\varpi^{-1}k^{\frac{2}{n}}-\varpi_0 \le \lambda_k^b\le \varpi k^{\frac{2}{n}},\quad k\ge 1.
\end{equation}

\subsection{Regularity of eigenfunctions}

In this subsection, $\Omega$ is of class $C^{1,1}$. Recall that
\[
\mathcal{A}_+:=\{a\in W_\ast ^{1,m}(\mathbb{R}^n,\mathbb{R}^n),\; \|a_{|\Omega}\|_{0,n}\le \mathfrak{c},\; \|a\|_{1,m}\le \tilde{\mathfrak{c}}\}.
\]
For $b=(a,V)\in \mathcal{A}_+\times \mathcal{V}$, define the magnetic Schr\"odinger operator $\Delta_b$ as follows
\[
\Delta_b
=\nabla_a\cdot \nabla_a+V=\Delta +2ia\cdot \nabla+V_a,
\]
where
\[
V_a:=i\mathrm{div}(a) -|a|^2+V.
\]

Let $k\ge 1$. From \eqref{5} with $u=\phi_k^b$ and \eqref{9} with $v=\phi_k^b$, we obtain
\[
\frac{\mathfrak{c}_-^2}{2}\|\nabla \phi_k^b\|_{0,2}^2\le \mathbf{c}+\lambda_k^b.
\]
Hence
\begin{equation}\label{10}
\|\phi_k^b\|_{0,p}+ \|\nabla \phi_k^b\|_{0,2}\le \mathbf{c}(1+|\lambda_k^b|)^{\frac{1}{2}}.
\end{equation}

Again, from \eqref{9}, we get that $-\Delta_b \phi_k^b=\lambda_k^b\phi_k^b$ in the distributional sense and then
\[
-\Delta \phi_k^b=2ia\cdot \nabla \phi_k^b+ (V -|a|^2+i\mathrm{div}(a))\phi_k^b+\lambda_k^b\phi_k^b.
\]
That is we have
\begin{align*}
&-\Delta \Re \phi_k^b=-2a\cdot \nabla \Im \phi_k^b+ (V -|a|^2)\Re \phi_k^b-\mathrm{div}(a)\Im \phi_k^b+\lambda_k^b\Re\phi_k^b\in L^p(\Omega),
\\
&-\Delta \Im \phi_k^b=2a\cdot \nabla \Re \phi_k^b+ (V -|a|^2)\Im \phi_k^b+\mathrm{div}(a)\Re \phi_k^b+\lambda_k^b\Im\phi_k^b\in L^p(\Omega).
\end{align*}

\cite[Theorem 9.15]{GT}, applied to both $\Re \phi_k^b$ and $\Im \phi_k^b$, gives $\phi_k^b\in W^{2,p}(\Omega)$. Given \eqref{10}, \cite[Theorem 9.14]{GT}, applied to both $\Re \phi_k^b$ and $\Im \phi_k^b$, leads to
\begin{equation}\label{11}
\|\phi_k^b\|_{2,p}\le \tilde{\mathbf{c}}(1+|\lambda_k^b|).
\end{equation}
Here and henceforth, $\tilde{\mathbf{c}}=\tilde{\mathbf{c}}(n,\Omega,V_0,\mathfrak{c},\tilde{\mathfrak{c}})>0$ will denote a generic constant.

\subsection{Resolvent estimates}

Let $b=(a,V)\in \mathcal{A}\times \mathcal{V}$. As usual, the resolvent set of $A^b$ is $\rho(A^b):=\mathbb{C}\setminus\sigma(A^b)$. From its definition, $A^b-\lambda:H_0^1(\Omega)\rightarrow H^{-1}(\Omega)$ is an isomorphism and
\[
R^b(\lambda):= (A^b-\lambda)^{-1}: H^{-1}(\Omega)  \rightarrow H_0^1(\Omega),\quad \lambda\in\rho(A^b),
\]
is called the resolvent of $A^b$. Therefore, we have
\[
S^b (R^b(\lambda )f,v)-\lambda (R^b(\lambda )f|v)=\langle f,\overline{v}\rangle ,\quad f\in H^{-1}(\Omega),\; v\in H_0^1(\Omega).
\]
In particular, the following formula holds
\begin{equation}\label{12}
\int_\Omega \left[|\nabla_aR^b(\lambda )f|^2-(V+\lambda)|R^b(\lambda )f|^2\right]dx=\int_\Omega f\overline{R^b(\lambda )f}dx, \quad f\in L^2(\Omega).
\end{equation}

Let
\[
\Sigma_0=\left\{\lambda \in \mathbb{C};\; |\Re \lambda|\ge 1,\; |\Im \lambda|\ge 1,\; |\Re \lambda||\Im \lambda|^{-1}\ge \frac{1}{2}\right\}.
\]

\begin{lemma}\label{lem1}
Let $b=(a,V)\in \mathcal{A}\times \mathcal{V}$, $\lambda\in \Sigma_0$ and $f\in L^2(\Omega)$. Then the following inequalities hold:
\begin{align}
&\|R^b(\lambda)f\|_{0,2}\le |\Im \lambda|^{-1}\|f\|_{0,2},\label{13}
\\
&\|\nabla R^b(\lambda)f\|_{0,2}\le \mathbf{c}|\Re \lambda|^{\frac{1}{2}}|\Im \lambda|^{-1}\|f\|_{0,2},\label{14}
\\
&\|\nabla R^b(\lambda)f\|_{0,2}\le \mathbf{c}|\Re \lambda||\Im \lambda|^{-1}\|f\|_{0,p}.\label{15}
\end{align}
\end{lemma}

\begin{proof}
Let $u:=R^b(\lambda)f$. It follows from \eqref{12}
\begin{align}
&\int_\Omega \left[|\nabla_au|^2-(V+\Re \lambda)|u|^2\right]dx=\Re \int_\Omega f\overline{u}dx,\label{16}
\\
&\Im \lambda \int_\Omega |u|^2dx=-\Im \int_\Omega f\overline{u}dx.\label{17}
\end{align}
Clearly, \eqref{17} implies \eqref{13} and, in light of \eqref{1}, \eqref{16} and \eqref{17} yield
\begin{align*}\label{18}
\mathfrak{c}_-^2\|\nabla u\|_{0,2}^2&\le \|\nabla_a u\|_{0,2}^2=\int_\Omega (V+\Re \lambda)|u|^2dx+\Re \int_\Omega f\overline{u}dx
\\
&\le |\Re \lambda|\|u\|_{0,2}^2+\|V|u|^2\|_{0,1}+\|f\overline{u}\|_{0,1}
\\
&\le ( |\Re \lambda||\Im \lambda|^{-1}+1)\|f\overline{u}\|_{0,1}+\|V|u|^2\|_{0,1}.
\end{align*}
Hence
\begin{equation}\label{18}
\mathfrak{c}_-^2\|\nabla u\|_{0,2}^2\le 3 |\Re \lambda||\Im \lambda|^{-1}\|f\overline{u}\|_{0,1}+\|V|u|^2\|_{0,1}.
\end{equation}
From Lemma \ref{lem8}, we have
\[
\|V|u|^2\|_{0,1}\le \epsilon \|\nabla u\|_{0,2}^2+c_\epsilon \|u\|_{0,2}^2,\quad \epsilon >0,
\]
where $c_\epsilon=c_\epsilon (n,\Omega,V_0,\epsilon )>0$ is a constant.

This inequality with $\epsilon =\frac{\mathfrak{c}_-^2}{2}$ in \eqref{18} gives
\begin{equation}\label{19}
\mathbf{c}\|\nabla u\|_{0,2}^2\le  |\Re \lambda||\Im \lambda|^{-1}\|f\overline{u}\|_{0,1}+ \|u\|_{0,2}^2,
\end{equation}
which, in combination with \eqref{13}, implies
\begin{align*}
\mathbf{c}\|\nabla u\|_{0,2}^2&\le  |\Re \lambda||\Im \lambda|^{-1}\|f\|_{0,2}\|u\|_{0,2}+ \|u\|_{0,2}^2
\\
&\le |\Im \lambda|^{-2}( |\Re \lambda|+1)\|f\|_{0,2}^2
\\
&\le 2|\Im \lambda|^{-2} |\Re \lambda|\|f\|_{0,2}^2.
\end{align*}
Then \eqref{14} follows.

Now, putting together \eqref{17} and \eqref{19}, we get
\begin{equation}\label{20}
\mathbf{c}\|\nabla u\|_{0,2}^2\le  |\Re \lambda||\Im \lambda|^{-1}\|f\overline{u}\|_{0,1}.
\end{equation}

Let $\epsilon >0$. Using H\"older's inequality, we obtain
\begin{align*}
\|f\overline{u}\|_{0,1}&\le \|f\|_{0,p}\|u\|_{0,p'}
\\
&\le \kappa \|f\|_{0,p}\|\nabla u\|_{0,2}
\\
&\le \frac{\kappa^2\epsilon}{2}\|\nabla u\|_{0,2}^2+\frac{1}{2\epsilon}\|f\|_{0,p}^2.
\end{align*} 
The last inequality with $\epsilon =\kappa^{-2}\mathbf{c} |\Re \lambda|^{-1}|\Im \lambda|$ in \eqref{20}, where $\mathbf{c}$ is as in \eqref{20}, yields \eqref{15}.
\end{proof}

By  slightly modifying the proof of Lemma \ref{lem1}, we obtain the following result. 

\begin{lemma}\label{lem2}
Let $b=(a,V)\in \mathcal{A}\times \mathcal{V}$, $\lambda\in \rho(A^b)$ and $f\in L^2(\Omega)$. Then there exists a constant $\mathbf{c}_\lambda=\mathbf{c}_\lambda(n,\Omega,b,\lambda)>0$ such that we have
\begin{align}
&\|R^b(\lambda)f\|_{0,2}\le \frac{1}{\mathrm{dist}(\lambda,\sigma(A^b))}\|f\|_{0,2},\quad f\in L^2(\Omega),\label{22}
\\
&\|\nabla R^b(\lambda)f\|_{0,2}\le \mathbf{c}_\lambda \|f\|_{0,2},\quad f\in L^2(\Omega),\label{23}
\\
&\|\nabla R^b(\lambda)f\|_{0,2}\le \mathbf{c}_\lambda \|f\|_{0,p},\quad f\in L^p(\Omega).\label{24}
\end{align}
\end{lemma}

We set
\[
\Sigma=\{\lambda_\tau=(\tau+i)^2;\; \tau \ge 2\}\; (\subset \Sigma_0).
\]

Then we have the following consequence of Lemma \ref{lem1}.

\begin{lemma}\label{lem3}
Let $b=(a,V)\in \mathcal{A}\times \mathcal{V}$ and $\lambda_\tau=(\tau+i)^2\in \Sigma$. Then the following inequalities hold:
\begin{align}
&2\|R^b(\lambda_\tau)f\|_{0,2}\le \tau ^{-1}\|f\|_{0,2},\quad f\in L^2(\Omega),\label{25}
\\
&\|R^b(\lambda_\tau)f\|_{0,p'}\le \mathbf{c}\tau \|f\|_{0,p},\quad f\in L^p(\Omega).\label{26}
\end{align}
\end{lemma}

Set
\begin{equation}\label{rtp}
p_\theta=\frac{2n}{n+2\theta},\quad p'_\theta=\frac{2n}{n-2\theta},\quad \theta\in [0,1].
\end{equation}
We verify that $p'_\theta$ is the conjugate of $p_\theta$. Given \eqref{25} and \eqref{26}, the following result is obtained as a consequence of Riesz-Thorin's theorem.

\begin{lemma}\label{lem4}
Let $b=(a,V)\in \mathcal{A}\times \mathcal{V}$, $\lambda_\tau=(\tau+i)^2\in \Sigma$. Then $R^b(\lambda_\tau)$ maps continuously $L^{p_\theta}(\Omega)$ into $L^{p'_\theta}(\Omega)$ and  we have
\begin{equation}\label{27}
\|R^b(\lambda_\tau)f\|_{0,p'_\theta}\le \mathbf{c}\tau^{-1+2\theta}\|f\|_{0,p_\theta},\quad f\in L^{p_\theta}(\Omega).
\end{equation}
\end{lemma}

\begin{lemma}\label{lem5}
There exists $\lambda_\ast=\lambda_\ast (n,\Omega,\mathfrak{c},V_0)$ such that for all $b=(a,V)\in \mathcal{A}\times \mathcal{V}$, $f\in L^2(\Omega)$ and $\lambda \in \rho(A^b)$ satisfying $-\Re \lambda \ge \lambda_\ast$ we have
\begin{align}
&\|R^b(\lambda)f\|_{0,2}\le 2|\Re \lambda|^{-1}\|f\|_{0,2},\label{28}
\\
&\|\nabla R^b(\lambda)f\|_{0,2}\le 2\mathfrak{c}^{-1}|\Re \lambda|^{-\frac{1}{2}}\|f\|_{0,2}.\label{29}
\end{align}
\end{lemma}

\begin{proof}
Let $\lambda \in \rho(A^b)$ such that $\Re \lambda\le 0$ and set $u:=R^b(\lambda)f$. From \eqref{12}, we have
\[
\int_\Omega \left[|\nabla_au|^2+|\Re \lambda||u|^2\right]dx=\int_\Omega V|u|^2dx+\Re \int_\Omega f\overline{u}dx, \quad f\in L^2(\Omega),
\]
which, in view of \eqref{1}, implies
\[
\mathfrak{c}_-^2\|\nabla u\|_{0,2}^2+|\Re \lambda|\|u\|_{0,2}^2\le \int_\Omega V|u|^2+\Re \int_\Omega f\overline{u}dx, \quad f\in L^2(\Omega).
\]
By combining this inequality with the following ones
\begin{align*}
& \|f\overline{u}\|_{0,1}\le \|u\|_{0,2}\|f\|_{0,2},
\\
&\|V|u|^2\|_{0,1}\le \frac{\mathfrak{c}_-^2}{2}\|\nabla u\|_{0,2}^2+\mathbf{c}\|u\|_{0,2}^2,
\end{align*}
we obtain
\[
\mathfrak{c}_-^2\|\nabla u\|_{0,2}^2+2(|\Re \lambda|-\mathbf{c})\|u\|_{0,2}^2\le 2\|f\|_{0,2}\|u\|_{0,2}.
\]
The expected inequalities follow by taking $\lambda_\ast=2\mathbf{c}$ in the inequality above.
\end{proof}

It follows from \cite[Theorem 9.14]{GT} that is exist $c_0=c_0(n,\Omega)>0$ and $\mu_0=\mu_0(n,\Omega)$ such that for all $w\in W^{2,q}(\Omega)\cap W_0^{1,q}(\Omega)$ and $\mu \ge \mu_0$ we have
\begin{equation}\label{30}
\|w\|_{2,q}\le c_0\|(\Delta -\mu)w\|_{0,q}.
\end{equation}
Define
\[
\Pi_1=\{\lambda \in \mathbb{C}\setminus \mathbb{R};\; -\Re \lambda \ge (\mu_0,\lambda_\ast),\; |\Re \lambda|^{-1}|\Im \lambda|\le 1\},
\]
where $\lambda_\ast$ is as in Lemma \ref{lem5} and $\mu_0$ is as above.

Henceforth, the generic constant $\tilde{\mathbf{c}}$ may depend on $s$.

\begin{lemma}\label{lem6}
Assume that $\Omega$ is of class $C^{1,1}$. For all $b=(a,V)\in \mathcal{A}_+\times \mathcal{V}$, $f\in L^2(\Omega)$ and $\lambda \in \Pi_1$ we have
\begin{equation}\label{31}
\|R^b(\lambda)f\|_{s,q}\le \tilde{\mathbf{c}}|\Re \lambda|^{-(1-\frac{s}{2})}\|f\|_{0,2}.
\end{equation}
\end{lemma}

\begin{proof}
Let $b=(a,V)\in \mathcal{A}_+\times \mathcal{V}$, $f\in L^2(\Omega)$ and $\lambda \in \Pi_1$. Set $u=v+iw=R^b(\lambda)(f)$. As
\[
\Delta u=-2ia\cdot \nabla u- V_a u-\lambda u-f \quad \mathrm{in}\; \Omega,
\]
where we recall that $V_a=V-|a|^2+i\mathrm{div}(a)$, we obtain 
\begin{align*}
&\Delta v +\Re \lambda v=2a\cdot \nabla w-\Re V_av+\Im V_a w+\Im \lambda w-\Re f,
\\
&\Delta w+\Re \lambda w=-2a\cdot \nabla v-\Re V_a w-\Im V_a v-\Im \lambda v-\Im f.
\end{align*}
Applying \eqref{30} to both $v$ and $w$, we get
\begin{align*}
\|u\|_{2,q}&\le \tilde{\mathbf{c}}\left(\|\nabla u\|_{0,q}+ (1+|\Im \lambda|)\|u\|_{0,q}+\|f\|_{0,q}\right)
\\
&\le \tilde{\mathbf{c}}\left(\|\nabla u\|_{0,2}+ (1+|\Im \lambda|)\|u\|_{0,2}+\|f\|_{0,2}\right)
\end{align*}
which, combined with \eqref{28} and \eqref{29}, gives 
\[
\|u\|_{2,q}\le \tilde{\mathbf{c}}\left(|\Re\lambda|^{-\frac{1}{2}}+ (1+|\Im \lambda|)|\Re \lambda|^{-1}+1\right)\|f\|_{0,2}.
\]
That is we have
\[
\|u\|_{2,q}\le \tilde{\mathbf{c}}\|f\|_{0,2}.
\]
This, \eqref{28} and the interpolation inequality
\[
\|w\|_{s,q}\le c\|w\|_{2,q}^{\frac{s}{2}}\|w\|_{0,q}^{1-\frac{s}{2}},\quad w\in W^{2,q}(\Omega)
\]
(e.g. \cite[Theorem 1.4.3.3]{Gr}), where $c=c(n,\Omega,s)$ is a constant, give \eqref{31}.
\end{proof}

\section{From BSD to a family of DtN maps}

\subsection{DtN maps}

In this subsection, $\Omega$ is of class $C^{1,1}$. 

\begin{proposition}\label{pro1}
Let $b=(a,V)\in \mathcal{A}_+\times \mathcal{V}$, $\lambda \in \rho(A^b)$, $f\in L^p(\Omega)$ and $\varphi\in W^{2-\frac{1}{p},p}(\Gamma)$. Then the BVP
\begin{equation}\label{32}
(-\Delta_b-\lambda)u=f\; \mathrm{in}\; \Omega,\quad u=\varphi\; \mathrm{on}\; \Gamma
\end{equation}
admits a unique solution $u(\lambda)(f,\varphi)\in W^{2,p}(\Omega)$. Furthermore, the following inequality holds.
\begin{equation}\label{33}
\|u(\lambda)(f,\varphi)\|_{2,p}\le \mathbf{c}_\lambda (\|f\|_{0,p}+\|\varphi\|_{2-\frac{1}{p},p}),
\end{equation}
where $\mathbf{c}_\lambda=\mathbf{c}_\lambda (n,\Omega, V_0,\mathfrak{c},\tilde{\mathfrak{c}},\lambda)>0$ is a constant.
\end{proposition}

\begin{proof}
In this proof, $\mathbf{c}_\lambda=\mathbf{c}_\lambda (n,\Omega, V_0,\mathfrak{c},\tilde{\mathfrak{c}},\lambda)>0$ will denote a generic constant.
Let $\Phi \in W^{2,p}(\Omega)$ such that $\Phi=\varphi$ on $\Gamma$ and
\begin{equation}\label{34}
\|\Phi\|_{2,p}\le c\|\varphi\|_{2-\frac{1}{p},p},
\end{equation}
where $c=c(n,\Omega)>0$ is a constant. 

Let $w=R^b(f+(\Delta_b+\lambda)\Phi)$. We have $w\in H_0^1(\Omega)$ and, using \eqref{24}, we obtain 
\begin{equation}\label{35}
\|\nabla w\|_{0,2}\le \mathbf{c}_\lambda \|f+(\Delta_b+\lambda)\Phi\|_{0,p}.
\end{equation}
In particular, we have
\[
-\Delta w=2ia\cdot \nabla w+ (i\mathrm{div}(a)-|a|^2+V)w+f+(\Delta_b+\lambda)\Phi:=F \in L^p(\Omega).
\]
Proceeding again as for \eqref{11}, we get from \eqref{30} 
\begin{equation}\label{36}
\|w\|_{2,p}\le c_0\|F+\mu_0w\|_{0,p},
\end{equation}
where $c_0$ and $\mu_0$ are as in \eqref{30}. Using \eqref{34} and \eqref{35} in \eqref{36}, we obtain
\begin{equation}\label{37}
\|w\|_{2,p}\le \mathbf{c}_\lambda(\|f\|_{0,p}+\|\varphi\|_{2-\frac{1}{p},p}).
\end{equation}
As $\lambda\in \rho(A^b)$, we verify that $u=w+\Phi$ is the unique solution of the BVP \eqref{32}. The proof is completed by noting that \eqref{33} follows readily from \eqref{37}.
\end{proof}

In what follows, for $b=(a,V)\in \mathcal{A}_+\times \mathcal{V}$, $\lambda \in \rho(A^b)$ and $\varphi\in W^{2-\frac{1}{p},p}(\Gamma)$, $u^b(\lambda)(\varphi)$ will denote the solution of \eqref{32} with $f=0$.

Before defining our family of DtN maps, we establish the following lemma.

\begin{lemma}\label{lem7}
Let $a\in W^{1,m}(\Omega,\mathbb{R}^n)$ such that $\|a\|_{1,m}\le \tilde{\mathfrak{c}}$ and $u\in W^{2,p}(\Omega)$. Then $a\cdot \nu u\in L^p(\Gamma)$ and
\begin{equation}\label{38}
\|a\cdot \nu u\|_{0,p}\le c\|u\|_{2,p},
\end{equation}
where $c=c(n,\Omega,\tilde{\mathfrak{c}})>0$ is a constant.
\end{lemma}

\begin{proof}
Let $\nu_e \in C^{0,1}(\overline{\Omega},\mathbb{R}^n)$ be an extension of $\nu$ (such an extension exists according to \cite[Theorem 4.2.3]{CMN}). As $C^{0,1}(\overline{\Omega},\mathbb{R}^n)$ is continuously embedded in $W^{1,\infty}(\Omega,\mathbb{R}^n)$  (by Radmacher's theorem), we have $v:=a\cdot \nu_e \in W^{1,m}(\Omega,\mathbb{C}^n)$. Let $1\le j\le n$. As $\partial_jv\in L^m(\Omega)$ and $u\in L^{p'}(\Omega)$, we obtain from \eqref{41} that $\partial_jvu\in L^p(\Omega)$. On the other hand, as $v\in L^n(\Omega)$ and $\partial_ju\in L^2(\Omega)$, again from \eqref{41}, we get $v\partial_j u\in L^p(\Omega)$. Therefore, $vu\in W^{1,p}(\Omega)$. The proof is completed by using the continuity of the trace map $w\in W^{1,p}(\Omega)\mapsto w_{|\Gamma}\in L^p(\Gamma)$.
\end{proof}

For $b=(a,V)\in \mathcal{A}_+\times \mathcal{V}$,  the family of DtN maps $(\Lambda^b(\lambda))_{\lambda \in \rho(A^b)}$ defined as follows:
\[
\Lambda^b(\lambda):\varphi \in W^{2-\frac{1}{p},p}(\Gamma)\mapsto \partial_{\nu_a}u^b(\lambda)(\varphi):=\nabla_au^b(\lambda)(\varphi)\cdot \nu\in L^p(\Gamma).
\]
Inequality \eqref{38} shows that the mapping $\Lambda^b(\lambda)$ defines a bounded operator.

In the remaining part of this text, we use the notation $X^{(j)}:=\frac{d^j}{d\lambda^j}X$ with the usual convention $X^{(0)}:=X$. 

Let $b=(a,V)\in \mathcal{A}_+\times \mathcal{V}$. By mimicking the proof of the result in \cite[Section 2.3]{Ch1}, we verify that
\[
\lambda\in \rho(A^b)\mapsto u^b(\lambda )(\varphi)\in W^{2,p}(\Omega)
\]
is holomorphic, and hence
\[
\lambda\in \rho(A^b)\mapsto \Lambda ^b(\lambda )(\varphi)\in L^p(\Gamma)
\]
is also holomorphic. Furthermore, $(u^b)^{(j)}(\lambda )(\varphi)\in W^{2,p}(\Omega)\cap W_0^{1,p}(\Omega)$, $j\ge 1$ an integer, satisfies
\begin{equation}\label{39}
(-\Delta_b -\lambda)(u^b)^{(j)}(\lambda )(\varphi)=j(u^b)^{(j-1)}(\lambda )(\varphi),
\end{equation}
and the following formula holds
\begin{equation}\label{40}
(u^b)^{(j)}(\lambda )(\varphi)=-j!\sum_{k\ge 1}\frac{\langle \varphi ,\psi^b_k\rangle}{(\lambda_k^b-\lambda)^{j+1}}\phi_k^b.
\end{equation}
Here and henceforth, $\langle \cdot,\cdot \rangle$ denotes the duality pairing between $L^\alpha(\Gamma)$, $1<\alpha<\infty$ and its dual $L^{\alpha'}(\Gamma)$ with $\alpha'=\frac{\alpha}{\alpha-1}$.

\subsection{BSD determines uniquely a family of DtN maps}

In this subsection,  $\Omega$ is again of class $C^{1,1}$. In what follows, 
\[
\mathcal{B}:=W^{2-\frac{1}{p},p}(\Gamma)\cap L^{q'}(\Gamma)
\]
 will be endowed with the norm
\[
\|\varphi\|:=\|\varphi\|_{2-\frac{1}{p},p}+\|\varphi\|_{0,q'},\quad \varphi \in \mathcal{B}.
\]
We recall that  $s\in ]1+\frac{1}{q},2[$ and $s_1=\frac{3s-2}{2+s}$. Henceforth, the generic constant $\mathbf{c}$ may depend on $s$.

\begin{lemma}\label{lem9}
There exists $\mu_\ast=\mu_\ast(n,\Omega,V_0,\mathfrak{c},s)$ such that for all $b=(a,V)\in \mathcal{A}_+\times \mathcal {V}$, $\varphi\in \mathcal{B}$ and $\lambda \in \mathbb{C}\setminus \mathbb{R}$ satisfying $-\Re \lambda \ge \mu_\ast$ the following inequalities hold:
\begin{align}
&\|u^b(\lambda)(\varphi)\|_{0,2}\le \mathbf{c}|\Re \lambda|^{-\frac{1}{2}}\|u^b(\lambda)(\varphi)\|_{s,q}^{\frac{1}{2}}\|\varphi\|^{\frac{1}{2}},\label{43}
\\
&\|u^b(\lambda)(\varphi)\|_{1,2}\le \mathbf{c}\|u^b(\lambda)(\varphi)\|_{s,q}^{\frac{1}{2}}\|\varphi\|^{\frac{1}{2}}.\label{44,1}
\end{align}
\end{lemma}

\begin{proof}
Let $b=(a,V)\in \mathcal{A}_+\times \mathcal {V}$, $\varphi\in \mathcal{B}$ and $\lambda \in \mathbb{C}\setminus \mathbb{R}$ such that $\Re \lambda \le 0$. Set $u:=u^b(\lambda)(\varphi)$. Applying Green's formula, we get
\[
\int_{\Omega}|\nabla_au|^2dx-\int_{\Omega}(V+\lambda)|u|^2dx=\int_\Gamma \partial_{\nu_a}u\overline{\varphi} d\sigma .
\]
Hence
\begin{equation}\label{*1}
\int_{\Omega}|\nabla_au|^2dx-\int_{\Omega}(V+\Re \lambda)|u|^2dx=\int_\Gamma \Re (\partial_{\nu_a}u\overline{\varphi })d\sigma .
\end{equation}
In view of Lemma \ref{lem8}, we have
\begin{equation}\label{*2}
\|V|u|^2\|_{0,1}\le \frac{1}{2}\|\nabla_au\|_{0,2}^2+\mathbf{c}\|u\|_{0,2}^2.
\end{equation}
Combining \eqref{*1} and \eqref{*2}, we obtain
\[
\frac{1}{2}\|\nabla_au\|_{0,2}^2+(|\Re \lambda|-\mathbf{c})\|u\|_{0,2}^2\le \|\partial_{\nu_a}u\varphi \|_{0,1}\le \|\partial_{\nu_a}u\|_{0,q}\|\varphi\|_{0,q'}.
\]
Given that the trace map $w\in W^{s,q}(\Omega) \mapsto \partial_{\nu_a}w\in L^q(\Gamma)$ is bounded, we derive
\[
\frac{1}{2}\|\nabla_au\|_{0,2}^2+(|\Re \lambda|-\mathbf{c})\|u\|_{0,2}^2\le c\|u\|_{s,q}\|\varphi\|_{0,q'},
\]
where $c=c(n,\Omega,s)$ is a constant. \eqref{43} and \eqref{44,1} follow from this inequality.
\end{proof}

$\lambda_\ast$ being as in Lemma \ref{lem5}, $\mu_\ast$ being as in Lemma \ref{lem9} and $\mu_0$ being as in \eqref{30}, we define
\[
\Pi_2:=\{ \lambda\in \mathbb{C}\setminus \mathbb{R};\; -\Re \lambda \ge \max(\lambda_\ast,\mu_\ast,\mu_0),\; |\Re \lambda|^{-1}|\Im \lambda|\le 1\}\; (\subset \Pi_1).
\]

\begin{lemma}\label{lem11}
For all $b=(a,V)\in \mathcal{A}_+\times \mathcal {V}$, $\varphi\in \mathcal{B}$ and $\lambda \in \Pi_2$ we have 
\begin{align}
&\|u^b(\lambda)(\varphi)\|_{s,q}\le \mathbf{c}|\Re \lambda|^{s_1}\|\varphi\|, \label{51}
\\
&\|u^b(\lambda)(\varphi)\|_{0,2}\le \mathbf{c}|\Re \lambda|^{-\frac{1}{2}+\frac{s_1}{2}}\|\varphi\|, \label{51,1}
\\
&\|u^b(\lambda)(\varphi)\|_{1,2}\le \mathbf{c}|\Re \lambda|^{\frac{s_1}{2}}\|\varphi\|.\label{51,3}
\end{align}
\end{lemma}

\begin{proof}
According to \eqref{8}, there exists $\lambda_0\in \mathbb{R}$, independent of $b\in \mathcal{A}_+\times \mathcal {V}$ such that $\lambda_0\in \rho(A^b)$.  Let $b=(a,V)\in \mathcal{A}_+\times \mathcal {V}$, $\varphi\in \mathcal{B}$ and $\lambda \in \Pi_2$. By applying Proposition \ref{pro1} with $f=(\lambda-\lambda_0) u^b(\lambda)(\varphi)$ , we obtain
\[
\|u^b(\lambda)(\varphi)\|_{2,q}\le \mathbf{c}(|\lambda|\|u^b(\lambda)(\varphi)\|_{0,2}+\|\varphi\|).
\]
Hence,
\begin{equation}\label{48}
\|u^b(\lambda)(\varphi)\|_{2,q}\le \mathbf{c}(|\Re \lambda|\|u^b(\lambda)(\varphi)\|_{0,2}+\|\varphi\|),
\end{equation}
which, in combination with \eqref{43}, yields
\[
\|u^b(\lambda)(\varphi)\|_{2,q}\le \mathbf{c}(|\Re \lambda|^{\frac{1}{2}}\|u^b(\lambda)(\varphi)\|_{2,q}^{\frac{1}{2}}\|\varphi\|^{\frac{1}{2}}+\|\varphi\|).
\]
By applying Cauchy-Schwarz's inequality, we obtain
\[
\|u^b(\lambda)(\varphi)\|_{2,q}\le \frac{1}{2} \|u^b(\lambda)(\varphi)\|_{2,q} +\mathbf{c}|\Re \lambda|\|\varphi\|.
\]
That is we have
\[
\|u^b(\lambda)(\varphi)\|_{2,q}\le \mathbf{c}|\Re \lambda|\|\varphi\|.
\]
Consequently, we obtain
\[
\|u^b(\lambda)(\varphi)\|_{s,q}\le \mathbf{c}|\Re \lambda|^{\frac{s}{2}}\|\varphi\|^{\frac{s}{2}}\|u^b(\lambda)(\varphi)\|_{0,2}^{1-\frac{s}{2}},
\]
which, combined with \eqref{43}, yields
\begin{align*}
\|u^b(\lambda)(\varphi)\|_{s,q}&\le \mathbf{c}|\Re \lambda|^{\frac{s}{2}}\|\varphi\|^{\frac{s}{2}}\left[|\Re \lambda|^{-\frac{1}{2}}\|u^b(\lambda)(\varphi)\|_{s,q}^{\frac{1}{2}}\|\varphi\|^{\frac{1}{2}}\right]^{1-\frac{s}{2}}.
\\
&\le \mathbf{c}|\Re \lambda|^{\frac{3s-2}{4}}\|u^b(\lambda)(\varphi)\|_{s,q}^{\frac{2-s}{4}}\|\varphi\|^{\frac{2+s}{4}}.
\end{align*}
Therefore, we get
\[
\|u^b(\lambda)(\varphi)\|_{s,q}\le \mathbf{c}|\Re \lambda|^{\frac{3s-2}{s+2}}\|\varphi\|.
\]
We have thus proved \eqref{51}. To complete the proof, we observe that \eqref{51,1} and \eqref{51,3} are obtained from \eqref{51}, in combination with \eqref{43} and \eqref{44,1}.
\end{proof}

In what follows, for all integer $j\ge 0$, $\tilde{\mathbf{c}}_j=\tilde{\mathbf{c}}_j(n,\Omega, V_0,\mathfrak{c},\tilde{\mathfrak{c}},s,j)>0$ will denote a generic constant.

\begin{lemma}\label{lem10}
Let $b=(a,V)\in \mathcal{A}_+\times \mathcal{V}$, $\varphi\in \mathcal{B}$, $\lambda\in \Pi_2$ and $j\ge 1$ an integer. Then the following inequalities hold:
\begin{align}
&\|(u^b)^{(j)}(\lambda)( \varphi)\|_{0,2}\le \tilde{\mathbf{c}}_j|\Re \lambda|^{-j-\frac{1}{2}+\frac{s_1}{2}}\|\varphi\|,\label{46}
\\
&\|(u^b)^{(j)}(\lambda)( \varphi)\|_{1,2}\le \tilde{\mathbf{c}}_j|\Re \lambda|^{-j+\frac{s_1}{2}}\|\varphi\|,\label{47.0}
\\
&\|(u^b)^{(j)}(\lambda)( \varphi)\|_{s,q}\le \tilde{\mathbf{c}}_j|\Re \lambda|^{-j-\frac{1}{2}+\frac{s+s_1}{2}}\|\varphi\|.\label{45}
\end{align}
\end{lemma}

\begin{proof}
From \eqref{39}, we have
\[
(u^b)^{(j)}(\lambda)( \varphi)=jR^b(\lambda)((u^b)^{(j-1)}(\lambda)( \varphi)),\quad j\ge 1.
\]
In view of \eqref{28}, \eqref{29} and \eqref{31}, we derive
\begin{align}
&\|(u^b)^{(j)}(\lambda)( \varphi)\|_{0,2}\le \tilde{\mathbf{c}}_j|\Re \lambda|^{-1}\|(u^b)^{(j-1)}(\lambda)( \varphi)\|_{0,2},\label{!1}
\\
&\|(u^b)^{(j)}(\lambda)( \varphi)\|_{1,2}\le \tilde{\mathbf{c}}_j|\Re \lambda|^{-\frac{1}{2}}\|(u^b)^{(j-1)}(\lambda)( \varphi)\|_{0,2},\label{!2}
\\
&\|(u^b)^{(j)}(\lambda)( \varphi)\|_{2,q}\le \tilde{\mathbf{c}}_j|\Re\lambda|^{-(1-\frac{s}{2})}\|(u^b)^{(j-1)}(\lambda)( \varphi)\|_{0,2}.\label{!3}
\end{align}
By induction in $j$, we get
\begin{equation}\label{!4}
\|(u^b)^{(j)}(\lambda)( \varphi)\|_{0,2}\le \tilde{\mathbf{c}}_j|\Re \lambda|^{-j}\|(u^b)^{(0)}(\lambda)( \varphi)\|_{0,2}.
\end{equation}
This and  \eqref{51,1} imply \eqref{!1}. Finally, \eqref{46} in \eqref{!2} and \eqref{!3} give \eqref{47.0} and \eqref{45}.
\end{proof}

Recall that $\varsigma= \frac{n(r-2)}{2r}$. From \eqref{46} and \eqref{47.0}, we have
\begin{align}
\|(u^b)^{(j)}(\lambda)( \varphi)\|_{0,r}&\le\|(u^b)^{(j)}(\lambda)( \varphi)\|_{0,2}^{1-\varsigma}\|(u^b)^{(j)}(\lambda)( \varphi)\|_{0,p'}^\varsigma \label{rr}
\\
&\le \tilde{\mathbf{c}}|\Re \lambda|^{-j-\frac{1}{2}+\frac{\varsigma+s_1}{2}}\|\varphi\|.\nonumber
\end{align}

We now state a useful Caccioppoli type lemma, the proof of which is provided at the end of this section.
\begin{lemma}\label{lemca}
Let $\tilde{\Omega}\Subset \Omega$, $b\in \mathcal{A}\times \mathcal{V}$, $\Re\lambda \le 0$. For all $w\in H^1(\Omega)$ satisfying $(-\Delta_b-\lambda)w=0$ in $\mathscr{D}'(\Omega)$, we have
\begin{equation}\label{ca1}
\int_{\tilde{\Omega}}|\nabla w|^2dx \le c \int_\Omega |w|^2dx,
\end{equation}
where $c=c(n,\Omega,\tilde{\Omega},V_0,\mathfrak{c})>0$ is a constant.
\end{lemma}

Recall that $\mathscr{B}$ denotes the set of pairs 
\[
(b_1,b_2)=((a_1,V_1),(a_2,V_2))\in [\mathcal{A}_+\times \mathcal{V}]^2
\] 
satisfying $\mathrm{supp}(a_1-a_2)\in K$ and
\[
|a_1-a_2|\le \mathbf{a},\quad |i\mathrm{div}(a_1-a_2)-|a_1|^2+|a_2|^2+ V_1-V_2|\le W_0.
\]
where $\mathbf{a}\in L_\ast^\infty (\mathbb{R}^n,\mathbb{R})$ and $W_0\in L_\ast^\mathfrak{a}(\mathbb{R}^n,\mathbb{R})$ are non negative and non identically zero.

In the remaining part of this section, $\mathbf{c}_j^\ast=\mathbf{c}_j^\ast (n,\Omega, \mathfrak{c},\tilde{\mathfrak{c}},\mathbf{a},V_0,W_0,s,r,j)>0$ is a generic constant, where $j\ge 0$ is an integer.

\begin{lemma}\label{lem12}
Let $(b_1,b_2)=((a_1,V_1), (a_2,V_2))\in \mathscr{B}$, $\lambda \in \Pi_2$ and $\varphi\in \mathcal{B}$. Then the following inequalities hold for all integer $j\ge 0$:
\begin{align}
&\|(u^{b_1})^{(j)}(\lambda)(\varphi)-(u^{b_2})^{(j)}(\lambda)(\varphi)\|_{s,q}\le \mathbf{c}_j^\ast|\Re \lambda|^{-(j+\frac{3}{2})+\frac{\varsigma+s+s_1}{2}}\|\varphi\|, \label{52}
\\
&\|(u^{b_1})^{(j)}(\lambda)(\varphi)-(u^{b_2})^{(j)}(\lambda)(\varphi)\|_{0,2}\le \mathbf{c}_j^\ast|\Re \lambda|^{-(j+\frac{3}{2})+\frac{\varsigma+s_1}{2}}\|\varphi\|. \label{53}
\end{align}
\end{lemma}

\begin{proof}
We have $u(\lambda):=u^{b_1}(\lambda)(\varphi)-u^{b_2}(\lambda)(\varphi)\in W^{2,q}(\Omega)\cap W_0^{1,q}(\Omega)$ and satisfies
\begin{align*}
-(\Delta_{b_1}+\lambda)u(\lambda)&=-(\Delta_{b_2}-\Delta_{b_1})u^{b_2}(\lambda)(\varphi)
\\
&=2ia\cdot \nabla u^{b_2}(\lambda)(\varphi)+ Wu^{b_2}(\lambda)(\varphi),
\end{align*}
where we set
\[
a:=a_1-a_2,\quad  W:=i\mathrm{div}(a_1-a_2)-|a_1|^2+|a_2|^2+ V_1-V_2.
\]
Since 
\[
u(\lambda)=R^{b_1}(\lambda)[2ia\cdot \nabla u^{b_2}(\lambda)(\varphi)+Wu^{b_2}(\lambda)(\varphi)], 
\]
and $\frac{1}{2}=\frac{1}{\mathfrak{a}}+\frac{1}{r}$ , we obtain from \eqref{31} 
\begin{align*}
\|u(\lambda)\|_{s,q}&\le \tilde{\mathbf{c}}|\Re \lambda|^{-(1-\frac{s}{2})}(\|2a\cdot \nabla u^{b_2}(\lambda)(\varphi)\|_{0,2}+\|Wu^{b_2}(\lambda)(\varphi)\|_{0,2})
\\
&\le \tilde{\mathbf{c}}|\Re \lambda|^{-(1-\frac{s}{2})}(\|2a\cdot \nabla u^{b_2}(\lambda)(\varphi)\|_{0,2}+\|W\|_{0,\mathfrak{a}}\|u^{b_2}(\lambda)(\varphi)\|_{0,r}).
\end{align*}
As $\mathrm{supp}(a)\subset K \subset \Omega$, it follows from \eqref{ca1} that \[\|a\cdot \nabla u^{b_2}(\lambda)(\varphi)\|_{0,2}\le \|a\|_{0,\infty}\|u^b(\lambda)(\varphi)\|_{0,2}.\]
Hence,
\[
\|u(\lambda)\|_{s,q}\le \tilde{\mathbf{c}} |\Re \lambda|^{-(1-\frac{s}{2})}\|u^{b_2}(\lambda)(\varphi)\|_{0,r}.
\]
This and \eqref{rr} imply
\[
\|u(\lambda)\|_{s,q}\le \tilde{\mathbf{c}} |\Re \lambda|^{-\frac{3}{2}+\frac{\varsigma+s+s_1}{2}}\|\varphi\|.
\]
That is, \eqref{52} holds for $j=0$. Similarly, using \eqref{28}, we obtain
\[
\|u(\lambda)\|_{0,2}\le \mathbf{c}_0^\ast|\Re \lambda|^{-\frac{3}{2}+\frac{\varsigma+s_1}{2}}\|\varphi\|.
\]
which is  \eqref{53} when $j=0$.

Next, taking the derivative, side by side, with respect to $\lambda$ in 
\[
-(\Delta_{b_1}+\lambda)u(\lambda)=2ia\cdot \nabla u^{b_2}(\lambda)(\varphi)+ Wu^{b_2}(\lambda)(\varphi),
\]
we find
\[
-(\Delta_{b_1}+\lambda)u^{(1)}(\lambda)= u(\lambda)+2ia\cdot \nabla u^{b_2}(\lambda)(\varphi)+ W(u^{b_2})^{(1)}(\lambda)(\varphi).
\]
Hence,
\[
u^{(1)}(\lambda)=R^{b_1}(\lambda)\left[u(\lambda)+2ia\cdot \nabla u^{b_2}(\lambda)(\varphi)+ W(u^{b_2})^{(1)}(\lambda)(\varphi)\right],
\]
and then
\begin{align*}
\|u^{(1)}(\lambda)\|_{s,q}&\le \tilde{\mathbf{c}}|\Re \lambda|^{-(1-\frac{s}{2})}\|u(\lambda)+2ia\cdot \nabla u^{b_2}(\lambda)(\varphi) + W(u^{b_2})^{(1)}(\lambda)(\varphi)\|_{0,2},
\\
&\le \mathbf{c}_1^\ast |\Re \lambda|^{-(1-\frac{s}{2})}(\|u(\lambda)\|_{0,2}+ \|(u^{b_2})^{(1)}(\lambda)(\varphi)\|_{0,r})
\end{align*}
and
\begin{align*}
\|u^{(1)}(\lambda)\|_{0,2}&\le \tilde{\mathbf{c}}|\Re \lambda|^{-1}\|-u(\lambda)+2ia\cdot \nabla u^{b_2}(\lambda)(\varphi)+ W(u^{b_2})^{(1)}(\lambda)(\varphi)\|_{0,2}
\\
&\le \mathbf{c}_1^\ast |\Re \lambda|^{-1}(\|u(\lambda)\|_{0,2}+ \|(u^{b_2})^{(1)}(\lambda)(\varphi)\|_{0,r}).
\end{align*}
We can argue as above to get from \eqref{47.0} with $j=1$ and \eqref{52} and \eqref{53} with $j=0$ 
 \begin{align*}
&\|u^{(1)}(\lambda)\|_{s,q} \le \mathbf{c}_1^\ast |\Re \lambda|^{-\frac{5}{2}+\frac{\varsigma+s+s_1}{2}}\|\varphi\|,
\\
&\|u^{(1)}(\lambda)\|_{0,2} \le \mathbf{c}_1^\ast |\Re \lambda|^{-\frac{5}{2}+\frac{\varsigma+s_1}{2}}\|\varphi\|.
\end{align*}
These establish that \eqref{52} and \eqref{53} when $j=1$. The proof of \eqref{52} and \eqref{53} is then  completed by induction in $j$.
\end{proof}

Recall that $\sigma=\frac{3}{\upsilon}$, where $\upsilon=\frac{3-s-s_1-\varsigma}{2}\, (>0)$, and define
\[
\Pi=\{ \lambda =-\tau^\sigma+2i\tau \in \Pi_2\}.
\]

The following corollary is a direct consequence of \eqref{52}.

\begin{corollary}\label{cor1}
Let $(b_1,b_2)=((a_1,V_1), (a_2,V_2))\in \mathscr{B}$, $\lambda \in \Pi$ and $\varphi\in \mathcal{B}$. Then we have for all integer $j\ge 0$
\begin{equation}\label{55}
\|(\Lambda^{b_1})^{(j)}(\lambda )(\varphi)-(\Lambda^{b_2})^{(j)}(\lambda )(\varphi)\|_{0,q}\le \mathbf{c}_j^\ast\tau^{-3-j\sigma}\|\varphi\|.
\end{equation}
\end{corollary}

Recall that
\[
\delta (b_1,b_2):=\sum_{k\ge 1}k^{-\frac{2\ell}{n}}\left[|\lambda_k^{b_1}-\lambda_k^{b_2}|+\|\psi_k^{b_1}-\psi_k^{b_2}\|_{0,q}\right],
\]
and set
\[
\Sigma_\ast:=\{ \lambda =(\tau+i)^2;\; \tau \ge \tau_\ast\},\quad \tau_\ast\ge 2.
\]

In view of \eqref{8} and \eqref{55}, we can proceed exactly as in the proof of \cite[Proposition 3.1]{Ch1} to establish the following result.

\begin{proposition}\label{pro2}
 There exists $\tau_\ast=\tau_\ast(n,\Omega,V_0,\mathfrak{c})$ such that for all $(b_1,b_2)=((a_1,V_1), (a_2,V_2))\in \mathscr{B}$  and $\lambda=(\tau+i)^2\in \Sigma_\ast$ we have
 \begin{equation}\label{56}
 \|\Lambda^{b_1}(\lambda )(\varphi)-\Lambda^{b_2}(\lambda )(\varphi)\|_{0,q}\le \mathbf{c}^\ast\left(\tau^{-3}+\tau^{(\sigma+1)\ell +2}\delta(b_1,b_2)\right)\|\varphi\|,
 \end{equation}
 where $\mathbf{c}^\ast=\mathbf{c}^\ast (n,\Omega, \mathfrak{c},\tilde{\mathfrak{c}},\mathbf{a},V_0,W_0,s,r)>0$ is a constant.
\end{proposition}

\begin{proof}[Proof of Lemma \ref{lemca}]
Let $w\in H^1(\Omega)$, $b=(a,V)\in \mathcal{A}\times \mathcal{V}$ satisfying $(-\Delta_b-\lambda)w=0$ in $\mathscr{D}'(\Omega)$. Then, using that $C_0^\infty(\Omega)$ is dense in $H_0^1(\Omega)$, we obtain
\begin{equation}\label{ca2}
0=\int_\Omega \nabla_a w \cdot \overline{\nabla_a v}dx-\int_{\Omega}(\lambda +V)w\overline{v}dx,\quad v\in H_0^1(\Omega).
\end{equation}
Let $\chi \in C_0^1(\Omega)$ satisfying $0\le \chi\le 1$ and $\chi=1$ in a neighborhood of $\tilde{\Omega}$. Applying \eqref{ca2} with $v=\chi^2w$, we get
\[
\int_\Omega \chi^2 |\nabla_a w|^2dx+|\Re\lambda|\int_\Omega \chi^2|w|^2=-\Re \int_\Omega 2\chi \overline{w}\nabla_a w \cdot \nabla \chi dx+\int_\Omega \chi^2 V|w|^2.
\]

Let $\epsilon>0$. In the remainder of this proof $c_\epsilon=c_\epsilon(n,\Omega,\tilde{\Omega},V_0,\mathfrak{c})>0$ is a generic constant.

Lemma \ref{lem8} applied to the second term of the right hand side of the inequality above and Cauchy-Schwarz's inequality applied to the first term of inequality above yields
\begin{equation}\label{ca3}
\int_\Omega \chi^2 |\nabla_a w|^2dx\le \epsilon \int_\Omega \chi^2 |\nabla_a w|^2dx+c_\epsilon\int_\Omega |w|^2dx+\epsilon\int_{\Omega}|\nabla (\chi w)|^2dx.
\end{equation}
In view of \eqref{1}, modifying $\epsilon$, we get
\[
\int_\Omega \chi^2 |\nabla_a w|^2dx\le \epsilon \int_\Omega \chi^2 |\nabla_a w|^2dx+c_\epsilon\int_\Omega |w|^2dx.
\]
Taking $\epsilon=\frac{1}{2}$ in this inequality, we obtain
\[
\int_\Omega \chi^2 |\nabla_a w|^2dx\le c\int_\Omega |w|^2dx,
\]
where $c=c(n,\Omega,\tilde{\Omega},V_0,\mathfrak{c})>0$ is a constant. In view of \eqref{1}, \eqref{ca1} follows.
\end{proof}

\section{The inverse spectral problem}

In this section, $\Omega$ is of class $C^{1,1}$.

\subsection{Integral formula}

For $b=(a,V)\in \mathcal{A}_0\times \mathcal{V}$, recall that
\[
\Delta_b=\Delta +2ia\cdot \nabla +V_a,
\]
where
\[
V_a=i\mathrm{div}(a)-|a|^2+V.
\]

We will use the following Green's formula:
\[
\int_\Omega \Delta_bu\overline{v}dx=\int_\Omega u\overline{\Delta_bv}dx+\int_\Gamma \partial_{\nu_a}u\overline{v}d\sigma-\int_\Gamma u\overline{\partial_{\nu_a}v}d\sigma,
\]
where $\partial_{\nu_a}:=\nu\cdot \nabla_a$. Of course, this formula is valid provided that each of its terms is well defined.

Let $b=(a,V)\in \mathcal{A}_+\times \mathcal{V}$, $\omega \in \mathbb{S}^{n-1}$. Then define
\[
\psi (x)=\psi (a,\omega)(x)=-\int_{-\infty}^0\omega \cdot a(x+s\omega)ds.
\]
Then 
\begin{align*}
\omega\cdot \nabla \psi (x)&=-\int_{-\infty}^0\omega \cdot \nabla [\omega\cdot a(x+s\omega)]ds
\\
&=-\int_{-\infty}^0 \frac{d}{ds}[\omega\cdot a(x+s\omega)]ds=-\omega \cdot a(x).
\end{align*}
That is, $\psi$ satisfies
\[
(\nabla \psi+ a)\cdot \omega=0.
\]

Let $\lambda\in \mathbb{C}\setminus (-\infty,0]$ and define  
\[
\varphi_+=\varphi_+(a,\omega,\lambda)=e^{i[\sqrt{\lambda}\omega\cdot x+\psi]},
\] 
with the standard choice of the branch of the square root. We have
\[
\nabla \varphi_+=  i\varphi_+(\sqrt{\lambda}\omega +\nabla \psi).
\]
Thus,
\begin{equation}\label{1.1}
2ia\cdot \nabla \varphi_+ =- 2 \varphi_+(\sqrt{\lambda}\omega\cdot a+\nabla \psi \cdot a)
\end{equation}
and
\begin{align*}
\Delta\varphi_+&= i\varphi_+\mathrm{div}(\sqrt{\lambda}\omega +\nabla \psi)-\varphi_+(\sqrt{\lambda}\omega +\nabla \psi)\cdot(\sqrt{\lambda}\omega +\nabla \psi)
\\
&= i\varphi_+ \Delta \psi - \varphi_+ (\lambda+2\sqrt{\lambda}\nabla \psi \cdot \omega+|\nabla\psi|^2).
\end{align*}
As $\nabla \psi \cdot \omega=-a\cdot \omega$, we obtain
\begin{equation}\label{1.2}
\Delta \varphi_+=\varphi_+( i\Delta \psi-\lambda +2\sqrt{\lambda}a\cdot \omega -|\nabla\psi|^2).
\end{equation}
Putting together \eqref{1.1} and \eqref{1.2},  we get
\[
\Delta \varphi_+ +2ia\cdot \nabla \varphi_+= \varphi_+( i\Delta \psi-\lambda -|\nabla\psi|^2-2\nabla \psi\cdot a)
\]
and then
\[
(\Delta_b+\lambda) \varphi_+=g_+\varphi_+,
\]
where
\[
g_+=g_+(b,\omega):= i\Delta \psi  -|\nabla\psi|^2-2\nabla \psi\cdot a+V_a\in L^m(\Omega).
\]

Also,  define
\[
\varphi_-=\varphi_-(a,\omega,\lambda)=e^{i\big[\, \overline{\sqrt{\lambda}}\omega\cdot x+\psi\big]}.
\] 
Calculations similar to those we performed above show that
\[
(\Delta_b+\overline{\lambda}) \varphi_-=g_-\varphi_-,
\]
where
\[
g_-=g_-(b,\omega):= i\Delta \psi  -|\nabla\psi|^2-2\nabla \psi\cdot a+V_a\in L^m (\Omega).
\]

For $j=1,2$, let $b_j=(a_j,V_j)\in \mathcal{A}_+\times \mathcal{V}$, $\omega_j\in \mathbb{S}^{n-1}$ and $\alpha\in W^{2,\infty}(\mathbb{R}^n)$ satisfying 
\[
\omega_2\cdot\nabla\alpha(x)=0. 
\]

We use in the following the notations 
\begin{align*}
&\psi_1=\psi (a_1,\omega_1),\quad \varphi_1=\varphi_+(a_1,\omega_1,\lambda),\quad g_1=g_+(b_1,\omega),
\\
&\psi_2=\psi(a_2,\omega_2),\quad \varphi_2=\varphi_-(a_2,\omega_2,\lambda),\quad g_2=g_-(b_2,\omega).
\end{align*}

If $u_1:=u^{b_1}(\lambda)(\varphi_1)$, $j=1,2$, then Green's formula above gives
\begin{align}
&\int_\Gamma \Lambda^{b_1}(\lambda)( \varphi_1) \overline{\alpha \varphi_2}d\sigma=\int_\Gamma \partial_{\nu_{a_1}}u_1\overline{\alpha \varphi_2}d\sigma \label{1.3}
\\
&\hskip 1cm =-\int_\Omega u_1\overline{\Delta_{b_1}(\alpha \varphi_2)}dx-\lambda \int_\Omega u_1\overline{\alpha \varphi_2}dx+\int_\Gamma \varphi_1\overline{\partial_{\nu_{a_1}}(\alpha\varphi_2)}d\sigma \nonumber
\\
&\hskip 1cm =-\int_\Omega u_1\overline{(\Delta_{b_1}+\overline{\lambda})(\alpha \varphi_2)}dx+\int_\Gamma \varphi_1\overline{\partial_{\nu_{a_1}}(\alpha\varphi_2)}d\sigma .\nonumber
\end{align}
We have
\[
\partial_{\nu_{a_1}}(\alpha \varphi_2)=i\alpha \varphi_2(\overline{\sqrt{\lambda}}\omega_2 \cdot \nu+\partial_\nu \psi_2+a_1\cdot \nu)+\partial_{\nu}\alpha \varphi_2,
\]
which we rewrite in the form 
\begin{equation}\label{1.4}
\partial_{\nu_{a_1}}(\alpha \varphi_2)=\rho\varphi_2, 
\end{equation}
where
\[
\rho:=i\alpha (\overline{\sqrt{\lambda}}\omega_2 \cdot \nu+\partial_\nu \psi_2+a_1\cdot \nu)+\partial_{\nu}\alpha .
\]
On the other hand, using $\omega_2\cdot \nabla \alpha(x)=0$, we find
\[
\Delta_{b_1}(\alpha \varphi_2)=\alpha \Delta_{b_1}\varphi_2+(\Delta \alpha +2i(a_1+\nabla \psi_2)\cdot \nabla \alpha)\varphi_2
\]
and since
\begin{align*}
(\Delta_{b_1}+\overline{\lambda})\varphi_2&=(\Delta_{b_2}+\overline{\lambda})\varphi_2+\Delta_{b_1}\varphi_2-\Delta_{b_2}\varphi_2
\\
&=g_2\varphi_2+2i(a_1-a_2)\cdot\nabla \varphi_2+(V_{a_1}-V_{a_2})\varphi_2
\\
&=g_2\varphi_2-2(a_1-a_2)\cdot(\overline{\sqrt{\lambda}}\omega_2+\nabla\psi_2) \varphi_2+(V_{a_1}-V_{a_2})\varphi_2,
\end{align*}
we obtain
\begin{equation}\label{1.5}
(\Delta_{b_1}+\overline{\lambda})(\alpha \varphi_2)=f\varphi_2,
\end{equation}
where
\begin{align*}
f:=\Delta \alpha +2i(a_1&+\nabla \psi_2)\cdot \nabla \alpha
\\
&+[g_2 -2(a_1-a_2)\cdot(\overline{\sqrt{\lambda}}\omega_2+\nabla\psi_2) +V_{a_1}-V_{a_2}]\alpha.
\end{align*}
Using \eqref{1.4} and \eqref{1.5} in \eqref{1.3} yields
\begin{equation}\label{1.6}
\int_\Gamma \Lambda^{b_1}(\lambda)( \varphi_1) \overline{\alpha \varphi_2}d\sigma=-\int_\Omega u_1\overline{f\varphi_2}dx+\int_\Gamma \varphi_1\overline{\rho\varphi_2}d\sigma.
\end{equation}
Clearly, 
\begin{equation}\label{1.7}
u_1=\varphi_1+R^{b_1}(\lambda)((\Delta_{b_1}+\lambda)\varphi_1)=\varphi_1+R^{b_1}(\lambda)(g_1\varphi_1).
\end{equation} 

Putting \eqref{1.7} in \eqref{1.6} gives
\begin{align}
&\int_\Gamma \Lambda^{b_1}(\lambda)( \varphi_1) \overline{\alpha \varphi_2}d\sigma=-\int_\Omega \varphi_1\overline{f\varphi_2}dx\label{1.8}
\\
&\hskip 3cm - \int_\Omega R^{b_1}(\lambda)(g_1\varphi_1)\overline{f\varphi_2}dx+\int_\Gamma \varphi_1\overline{\rho\varphi_2}d\sigma.\nonumber
\end{align}

Next, let $v_1:=u^{b_2}(\lambda)(\varphi_1)$. Applying again Green's formula above, we find
\begin{align*}
&\int_\Gamma \Lambda^{b_2}(\lambda)( \varphi_1 ) \overline{\alpha \varphi_2}d\sigma=\int_\Gamma \partial_{\nu_{a_2}}v_1\overline{\alpha \varphi_2}d\sigma 
\\
&\hskip 1cm =-\int_\Omega v_1\overline{(\Delta_{b_2}+\overline{\lambda})(\alpha \varphi_2)}dx+\int_\Gamma \varphi_1\overline{\partial_{\nu_{a_2}}(\alpha\varphi_2)}d\sigma .
\end{align*}

Also, we have
\begin{align*}
(\Delta_{b_2}+\overline{\lambda})(\alpha \varphi_2)&=\alpha (\Delta_{b_2}+\overline{\lambda})\varphi_2+(\Delta \alpha +2i(a_2+\nabla \psi_2)\cdot \nabla \alpha)\varphi_2
\\
&=g_2\alpha \varphi_2+(\Delta \alpha +2i(a_2+\nabla \psi_2)\cdot \nabla \alpha)\varphi_2:=h\varphi_2.
\end{align*}
For further use, note that
\[
f=h+[-2(a_1-a_2)\cdot(\overline{\sqrt{\lambda}}\omega_2+\nabla\psi_2) +V_{a_1}-V_{a_2}]\alpha+2i(a_1-a_2)\cdot \nabla \alpha.
\]
If 
\[
\vartheta:= i\alpha (\overline{\sqrt{\lambda}}\omega_2 \cdot \nu+\partial_\nu \psi_2+a_2\cdot \nu)+\partial_{\nu}\alpha ,
\]
then we obtain
\begin{equation}\label{1.11}
\int_\Gamma \Lambda^{b_2}(\lambda)( \varphi_1 ) \overline{\alpha \varphi_2}d\sigma=-\int_\Omega v_1\overline{h\varphi_2}dx +\int_\Gamma \varphi_1\overline{\vartheta\varphi_2}d\sigma.
\end{equation}

We verify that $(\Delta_{b_2}+\lambda)\varphi_1=\hat{g}\varphi_1$, where
\[
\hat{g}=g_1 -2(a_2-a_1)\cdot(\sqrt{\lambda}\omega_1+\nabla\psi_1) +V_{a_2}-V_{a_1}.
\]
In consequence, we have
\[
v_1=\varphi_1+R^{b_2}(\lambda)(\hat{g}\varphi_1).
\]
This identity in \eqref{1.11} gives
\begin{align}
&\int_\Gamma \Lambda^{b_2}(\lambda)( \varphi_1) \overline{\alpha \varphi_2}d\sigma=-\int_\Omega \varphi_1\overline{h\varphi_2}dx\label{1.12}
\\
&\hskip 3cm - \int_\Omega R^{b_2}(\lambda)(\hat{g}\varphi_1)\overline{h\varphi_2}dx+\int_\Gamma \varphi_1\overline{\vartheta\varphi_2}d\sigma.\nonumber
\end{align}

Define
\begin{align*}
&\mathbf{d}:=\overline{f-h}=-2i(a_1-a_2)\cdot \nabla \overline{\alpha}+[-2(a_1-a_2)\cdot(\sqrt{\lambda}\omega_2+\overline{\nabla\psi_2}) +V_{a_1}-V_{a_2}]\overline{\alpha},
\\
&\zeta:=\overline{\rho-\vartheta}=-i(a_1-a_2)\cdot \nu \overline{\alpha}.
\end{align*}

Combining \eqref{1.8} and \eqref{1.12}, we obtain
\begin{align*}
&\int_\Gamma [\Lambda^{b_1}(\lambda)-\Lambda^{b_2}(\lambda)]( \varphi_1) \overline{\alpha \varphi_2}d\sigma=-\int_\Omega \mathbf{d}\varphi_1\overline{\varphi_2}dx
\\
&\hskip 1cm- \int_\Omega R^{b_1}(\lambda)(g_1\varphi_1)\overline{f\varphi_2}dx+\int_\Omega R^{b_2}(\lambda)(\hat{g}\varphi_1)\overline{h\varphi_2}dx+\int_\Gamma \zeta\varphi_1\overline{\varphi_2}d\sigma.
\end{align*}
When $a_1=a_2$ on $\Gamma$, the identity above becomes
\begin{align}
&\int_\Gamma [\Lambda^{b_1}(\lambda)-\Lambda^{b_2}(\lambda)]( \varphi_1) \overline{\alpha \varphi_2}d\sigma=-\int_\Omega \mathbf{d}\varphi_1\overline{\varphi_2}dx\label{1.13}
\\
&\hskip 2cm - \int_\Omega R^{b_1}(\lambda)(g_1\varphi_1)\overline{f\varphi_2}dx+\int_\Omega R^{b_2}(\lambda)(\hat{g}\varphi_1)\overline{h\varphi_2}dx.\nonumber
\end{align}

\subsection{Proof of Theorem \ref{thm1}}

In this subsection, $a\in \mathcal{A}_0$ is arbitrary fixed and $\mathbf{c}=\mathbf{c}(n,\Omega,V_0,a)>0$ will denote a generic constant. Let $\tau>1$, $\lambda=(\tau+i)^2$. For $\xi\in \mathbb{R}^n$, $|\xi|<\tau$, set
\[
\varrho=\sqrt{1-\frac{|\xi|^2}{4\tau^2}},\quad \omega_1=\varrho\omega-\frac{1}{2\tau}\xi,\quad \omega_2=\varrho\omega+\frac{1}{2\tau}\xi.
\]

Let 
\[
\psi_j:=\psi(a,\omega_j),\quad j=1,2,
\]
$\varphi_1=\varphi_+(a,\omega_1,\lambda)(x)$ and $\varphi_2=\varphi_-(a,\omega_2,\lambda)(x)$. Then 
\begin{align*}
\varphi_1\overline{\varphi_2}&= e^{i(\tau+i)\omega_1\cdot x+i\psi_1}e^{-i(\tau+i)\omega_2\cdot x-i\psi_2}
\\
&= e^{-i(1+\frac{i}{\tau})x\cdot\xi}e^{i(\psi_1-\psi_2)},
\end{align*}
that we rewrite in the form
\begin{equation}\label{o3.0}
\varphi_1\overline{\varphi_2}=e^{-ix\cdot\xi}+\Phi,
\end{equation}
where
\[
\Phi:=e^{-i(1+\frac{i}{\tau})x\cdot\xi}e^{i(\psi_1-\psi_2)}-e^{-ix\cdot\xi}=e^{-ix\cdot\xi}(e^{\frac{1}{\tau}x\cdot\xi+i(\psi_1-\psi_2)}-1).
\]

Also, we have
\begin{align*}
\psi_1-\psi_2&=\int_{-\infty}^0\omega_1 \cdot a(x+s\omega_1)ds-\int_{-\infty}^0\omega_2 \cdot a(x+s\omega_2 )ds
\\
&=\int_{-\infty}^0(\omega_1-\omega_2) \cdot a(x+s\omega_1)ds
\\
&\hskip 3cm +\int_{-\infty}^0\omega_2 \cdot [a(x+s\omega_1)-a(x+s\omega_2)]ds.
\\
&=-\frac{1}{\tau}\int_{-\infty}^0 \xi \cdot a(x+s\omega_1)ds
\\
&\hskip 3cm +\int_{-\infty}^0\omega_2 \cdot [a(x+s\omega_1)-a(x+s\omega_2)]ds,
\end{align*}
and then
\begin{equation}\label{o3.1}
\|\psi_1-\psi_2\|_{0,\infty}\le \mathbf{c}\frac{|\xi| }{\tau}.
\end{equation}
In view of \eqref{o3.1}, we verify that
\begin{equation}\label{o3.2}
\|\Phi\|_{0,\infty}\le \mathbf{c}\frac{|\xi| }{\tau}.
\end{equation}

For $j=1,2$, let $V_j\in \mathcal{V}$ and  $b_j=(a,V_j)$. We will apply formula \eqref{1.13} with $\alpha=1$ and $a_1=a_2=a$ , for which we have
\begin{align*}
&\mathbf{d}=V_1-V_2,
\\
&g_1=i\Delta \psi_1-|\nabla \psi_1|^2-2\nabla \psi_1\cdot a +i\mathrm{div}(a)-|a|^2+V_1,
\\
&g_2=i\Delta \psi_2-|\nabla \psi_2|^2-2\nabla \psi_2\cdot a +i\mathrm{div}(a)-|a|^2+V_2,
\\
&\hat{g}=g_1+(V_2-V_1),
\\
&h=g_2,
\\
&f=g_2+(V_1-V_2).
\end{align*}

We verify that
\begin{equation}\label{o3.3}
\|w\|_{0,t}\le \mathbf{c},\quad w\in\{\mathbf{d},g_1,g_2,\hat{g},h,f\},
\end{equation}
where we recall that $t=\frac{8n}{4n+1}$.

Until the end of this section, $\tilde{\mathbf{c}}=\tilde{\mathbf{c}}(n,\Omega,\Omega_0, a,V_0,W_0,K,s,r)>0$ will denote a generic constant.

As $t=p_{\frac{1}{8}}$, where $p_{\frac{1}{8}}$ is given by \eqref{rtp} when $\theta=\frac{1}{8}$, using that $\|\varphi_j\|_{0,\infty}=\|e^{|x|}\|_{0,\infty}$, $j=1,2$, and \eqref{o3.3}, we derive from \eqref{27}
\[
\left|\int_\Omega R^{b_1}(g_1\varphi_1)\overline{f\varphi_2}dx\right|\le \mathbf{c}\tau^{-\frac{3}{4}}\|g_1\varphi_1\|_{0,t}\|f\varphi_2\|_{0,t}\le \mathbf{c}\tau^{-\frac{3}{4}}.
\]
Similarly, we have
\[
\left|\int_\Omega R^{b_2}(\hat{g}\varphi_1)\overline{h\varphi_2}dx\right|\le \mathbf{c}\tau^{-\frac{3}{4}}.
\]
Consequently, we obtain
\begin{equation}\label{o4}
\left|\int_\Omega R^{b_1}(g_1\varphi_1)\overline{f\varphi_2}dx-\int_\Omega R^{b_2}(\hat{g}\varphi_1)\overline{h\varphi_2}dx\right|\le \mathbf{c}\tau^{-\frac{3}{4}}.
\end{equation}

Assume that $(b_1,b_2)$ satisfies the assumptions of Theorem \ref{thm1} and set $V=V_1-V_2\in L_\ast^n(\mathbb{R}^n,\mathbb{R}^n)$. Putting together \eqref{56}, \eqref{1.13}, \eqref{o3.3}, \eqref{o3.2} and \eqref{o4}, and using that
\[
\|\varphi_1\|\le \mathbf{c}\tau^2,\quad \|\varphi_2{_{|\Gamma}}\|_{0,q'}\le \mathbf{c},
\]
we obtain
\[
\mathbf{\tilde{c}}|\hat{V}(\xi)|\le \tau^{-\frac{3}{4}}+|\xi| \tau^{-1}+\tau^{(\sigma+1)\ell+4}\delta(b_1,b_2),\quad \tau \ge \tau_\ast,
\] 
where $\tau^\ast$ is as in Proposition \ref{pro2}.

Let $\tau \ge \tau_\ast$ and $0<\alpha<1$. Then
\begin{align}
&\mathbf{\tilde{c}}\int_{|\xi|\le \tau^{\alpha}}(1+|\xi|^2)^{-1}|\hat{V}(\xi)|^2d\xi\label{.1}
\\
&\hskip 2cm\le \tau^{-\frac{3}{2}+n\alpha}+\tau^{-2+(n+2)\alpha}+\tau^{2[(\sigma+1)\ell+4]+n\alpha}\delta(b_1,b_2)^2.\nonumber
\end{align}
On the other hand, in view of Parseval's inequality, we have
\begin{equation}\label{.2}
\int_{|\xi|> \tau^{\alpha}}(1+|\xi|^2)^{-1}|\hat{V}(\xi)|^2d\xi\le \tau^{-2\alpha}\|\hat{V}\|_{0,2}^2=\tau^{-2\alpha}\|V\|_{0,2}^2\le \mathbf{\tilde{c}}\tau^{-2\alpha}.
\end{equation}
As $\|V_1-V_2\|_{H^{-1}(\Omega)}\le \|V\|_{H^{-1}(\Omega)}$, putting together \eqref{.1} and \eqref{.2} to obtain
\begin{align}
&\mathbf{\tilde{c}}\|V\|_{H^{-1}(\Omega)}^2\label{.3}
\\
&\hskip 1cm\le \tau^{-2\alpha}+\tau^{-\frac{3}{2}+n\alpha}+\tau^{-2+(n+2)\alpha}+\tau^{2[(\sigma+1)\ell+4]+n\alpha}\delta(b_1,b_2)^2.\nonumber
\end{align}
Taking $\alpha=\frac{1}{2(n+2)}$ in \eqref{.3}, we obtain
\[
\mathbf{\tilde{c}}\|V\|_{H^{-1}(\Omega)}^2\le \tau^{-\frac{1}{n+2}}+\tau^{2[(\sigma+1)\ell+4]+\frac{n}{2(n+2)}}\delta(b_1,b_2)^2.
\]
That is, we have
\[
\mathbf{\tilde{c}}\|V\|_{H^{-1}(\Omega)}\le \tau^{-\frac{1}{2(n+2)}}+\tau^{(\sigma+1)\ell+4+\frac{n}{4(n+2)}}\delta(b_1,b_2).
\]
Minimizing, with respect to $\tau$, the right hand side of inequality above, we get
\[
\|V_1-V_2\|_{H^{-1}(\Omega)}\le \tilde{\mathbf{c}}\delta(b_1,b_2)^{\beta_0},
\]
where
\[
\beta_0=\frac{1}{2(n+2)\left((\sigma+1)\ell+\frac{17}{4}\right)}.
\]
Theorem \ref{thm1} is then proved.

\subsection{Proof of Theorem \ref{thm2}}

Let $(b_1,b_2)=((a_1,V_1), (a_2,V_2))\in \mathscr{B}_+$. Recall that $a:=a_1-a_2\in W_\ast^{2,\infty}(\mathbb{R}^n,\mathbb{R}^n)\cap W_\ast^{3,t}(\mathbb{R}^n,\mathbb{R}^n)$ and 
\[
\tilde{a}=(\tilde{a}^1,\ldots,\tilde{a}^n):=\chi_\Omega(a_1-a_2)\in H^1(\mathbb{R}^n,\mathbb{R}^n).
\]

In the remaining part of this subsection, $\tilde{\mathbf{c}}=\tilde{\mathbf{c}}(n,\Omega,\Omega_0,\mathbf{a}, V_0,W_0,\mathfrak{c},\tilde{\mathfrak{c}},\varepsilon, \rho)>0$ will denote a generic constant. 

Let $\xi\in \mathbb{R}^n$ and $\omega\in \mathbb{S}^{n-1}$ such that $\omega\cdot \xi=0$. Let $\tau>1$, $\lambda=(\tau+i)^2$. For $\xi\in \mathbb{R}^n$, $|\xi|<\tau$, set
\[
\varrho=\sqrt{1-\frac{|\xi|^2}{4\tau^2}},\quad \omega_1=\varrho\omega-\frac{1}{2\tau}\xi,\quad \omega_2=\varrho\omega+\frac{1}{2\tau}\xi,\quad \theta=\varrho \xi-\frac{|\xi|^2}{2\tau}\omega\in \omega_2^\bot.
\]

Using that $\varrho^2+\frac{|\xi|^2}{4\tau^2}=1$ and $(1-\varrho)^2\le 1-\varrho^2$, we find
\begin{align*}
|\xi -\theta|^2&=(1-\varrho)^2|\xi|^2+\frac{|\xi|^4}{4\tau^2}
\\
&=(1-\varrho)^2|\xi|^2+(1-\varrho^2)|\xi|^2
\\
&\le 2|\xi|^2(1-\varrho^2)=\frac{|\xi|^4}{2\tau^2}.
\end{align*}
Hence
\begin{equation}\label{o5}
|\xi-\theta|\le \frac{|\xi|^2}{\tau}.
\end{equation}
Similarly as above, we establish
\begin{equation}\label{o5.1}
|\omega-\omega_1|=|\omega-\omega_2|\le \frac{|\xi|}{\tau}.
\end{equation}

Let $\eta \in \mathbb{S}^{n-1}\cap \omega^\bot$ and  define
\begin{align*}
&\psi_j(x)=-\int_{-\infty}^0\omega_j\cdot a_j (x+s\omega_j)ds,\quad j=1,2.
\\
&\tilde{\psi}_2(x)=-\int_{-\infty}^0\omega_2\cdot a (x+s\omega_2)ds
\\
&\overline{\alpha}=e^{i\theta \cdot x}\eta \cdot \nabla [e^{i\tilde{\psi}_2}e^{-i\theta \cdot x}]=\eta\cdot (i\nabla \tilde{\psi}_2(x)-i\theta)e^{i\tilde{\psi}_2(x)},
\\
&\tilde{\psi}(x)=-\int_{-\infty}^0\omega\cdot a(x+s\omega)ds,
\\
&\overline{\beta}=e^{i\xi \cdot x}\eta \cdot \nabla [e^{i\tilde{\psi}(x)}e^{-i\xi \cdot x}]=\eta\cdot (i\nabla \tilde{\psi}(x)-i\xi)e^{i\tilde{\psi}(x)}.
\end{align*}

In what follows, we will use the following inequalities, which are straightforward to establish:
\begin{align}
&\|\alpha \|_{0,\infty}\le c_0(1+|\xi|),\label{o16}
\\
&\|\beta \|_{0,\infty}\le c_0(1+|\xi|).\label{o16.1}
\end{align}
Here and henceforth, $c_0=c_0(\Omega,\tilde{\mathfrak{c}})>0$ denotes a generic constant.

On the other hand, noting that 
\begin{align*}
&\nabla \overline{\alpha}=e^{i\tilde{\psi}_2}\nabla (\eta\cdot \nabla \tilde{\psi}_2)+i\overline{\alpha}\nabla \tilde{\psi}_2,
\\
&\Delta \overline{\alpha}=[\Delta (\eta \cdot \nabla \tilde{\psi}_2)+\nabla(\eta \cdot \nabla \tilde{\psi}_2)\cdot \nabla \tilde{\psi}_2]e^{i\tilde{\psi}_2}+i[\Delta \tilde{\psi}_2+|\nabla \tilde{\psi}_2|^2]\overline{\alpha},
\end{align*}
we obtain
\begin{equation}\label{o16.2}
\| \nabla \alpha\|_{0,\infty}+ \|\Delta \alpha \|_{0,t}\le c_0(1+|\xi|).
\end{equation}
Here, we recall again that $t=\frac{8n}{4n+1}$.

Also, in view of the decomposition
\[
\tilde{\psi}_2-\tilde{\psi}=\int_{\mathbb{R}}(\omega-\omega_2)\cdot a(x+s\omega)ds,
+\int_{\mathbb{R}}\omega_2\cdot (a(x+s\omega)-a(x+s\omega_2))ds,
\]
we obtain from \eqref{o5.1}
 \begin{equation}\label{o11}
 \|\tilde{\psi}_2-\tilde{\psi}\|_{0,\infty}\le c_0\frac{|\xi|}{\tau}.
 \end{equation}

 Since
 \[
e^{i\tilde{\psi}_2}-e^{i\tilde{\psi}}=i\int_0^1e^{i(\tilde{\psi}+s(\tilde{\psi}_2-\tilde{\psi}))}(\psi-\tilde{\psi})ds,
\]
\eqref{o11} yields
 \begin{equation}\label{o12}
\|e^{i\psi}-e^{i\tilde{\psi}}\|_{0,\infty}\le c_0\frac{|\xi|}{\tau}.
\end{equation}
 
As, for $1\le k\le n$,
\begin{align*}
\partial_k\tilde{\psi}_2(x)-\partial_k\tilde{\psi}(x)=\int_{\mathbb{R}}(\omega-\omega_2)&\cdot \partial_ka(x+s\omega)ds
\\
&+\int_{\mathbb{R}}\omega_2\cdot (\partial_ka (x+s\omega)-\partial_ka (x+s\omega_2))ds, 
\end{align*}
we obtain, similarly as above,
 \begin{equation}\label{o14}
 \|\nabla \psi -\nabla \tilde{\psi}\|_{0,\infty}\le c_0\frac{|\xi|}{\tau}.
 \end{equation}

In light of \eqref{o5}, \eqref{o12} and \eqref{o14}, we verify that 
\begin{equation}\label{o15}
\|\alpha -\beta\|_{0,\infty}\le c_0(1+|\xi|)\frac{|\xi|}{\tau}.
\end{equation}

In the present case, we have 
\[
\mathbf{d}=-2ia\cdot \nabla \overline{\alpha}+[-2\sqrt{\lambda}a\cdot \omega_2-2a\cdot\overline{\nabla\psi_2}+V_{a_1}-V_{a_2}]\overline{\alpha}
\]
and hence
\begin{align*}
&\frac{\mathbf{d}}{\sqrt{\lambda}}=-2a\cdot \omega_2\overline{\beta}+2a\cdot \omega_2\overline{\beta-\alpha}
\\
&\hskip 3cm -\frac{2i}{\sqrt{\lambda}}a\cdot \nabla \overline{\alpha} +\frac{1}{\sqrt{\lambda}}[-2a\cdot \overline{\nabla\psi_2}+V_{a_1}-V_{a_2}]\overline{\alpha}.
\\
&\hskip .5cm=-2a\cdot \omega\overline{\beta}+2a\cdot (\omega-\omega_2)\overline{\beta}+2a\cdot \omega_2\overline{\beta-\alpha}
\\
&\hskip 3cm -\frac{2i}{\sqrt{\lambda}}a\cdot \nabla \overline{\alpha}+\frac{1}{\sqrt{\lambda}}[-2a\cdot \overline{\nabla\psi_2}+V_{a_1}-V_{a_2}]\overline{\alpha}.
\end{align*}
From \eqref{o5.1}, \eqref{o16}, \eqref{o16.1}, \eqref{o16.2} and \eqref{o15}, we can write
\begin{equation}\label{o17}
\frac{\mathbf{d}}{\sqrt{\lambda}}=-2a\cdot \omega\overline{\beta}+\mathbf{r},
\end{equation}
where $\mathbf{r}$ satisfies 
\begin{equation}\label{o18}
\|\mathbf{r}\|_{0,1}\le \tilde{\mathbf{c}}\left( \frac{|\xi|}{\tau}+\frac{|\xi|^2}{\tau}\right).
\end{equation}
Here, we have taken into account that $\frac{|\xi|^2}{\tau^2}\le \frac{|\xi|}{\tau}$.

Next, the following successive identities
\begin{align*}
\varphi_1\overline{\varphi_2}&=e^{-i\xi\cdot x+\frac{1}{\tau}x\cdot \xi}e^{i(\psi_1-\psi_2)}
\\
&=e^{-i\xi\cdot x}e^{i\tilde{\psi}}+e^{-i\xi\cdot x}\left(e^{\frac{1}{\tau}x\cdot \xi}e^{i(\psi_1-\psi_2)}-e^{i\tilde{\psi}}\right)
\\
&=e^{-i\xi\cdot x}e^{i\tilde{\psi}}+e^{-i\xi\cdot x}\left(e^{\frac{1}{\tau}x\cdot \xi}-1\right)e^{i(\psi_1-\psi_2)}+e^{-i\xi\cdot x}\left(e^{i(\psi_1-\psi_2)}-e^{i\tilde{\psi}}\right)
\\
&=e^{-i\xi\cdot x}e^{i\tilde{\psi}}+e^{-i\xi\cdot x}\left(e^{\frac{1}{\tau}x\cdot \xi}-1\right)e^{i(\psi_1-\psi_2)}+e^{-i(\xi\cdot x-\tilde{\psi})}\left(e^{i(\psi_1-\psi_2-\tilde{\psi})}-1\right)
\end{align*}
show that
\begin{equation}\label{o19}
\varphi_1\overline{\varphi_2}=e^{-i\xi\cdot x}e^{i\tilde{\psi}}+\Phi_1+\Phi_2,
\end{equation}
where
\begin{align*}
&\Phi_1=e^{-i\xi\cdot x}\left(e^{\frac{1}{\tau}x\cdot \xi}-1\right)e^{i(\psi_1-\psi_2)},
\\
&\Phi_2=e^{-i(\xi\cdot x-\tilde{\psi})}\left(e^{i(\psi_1-\psi_2-\tilde{\psi})}-1\right).
\end{align*}
Using the decomposition
\begin{align*}
&\psi_1-\psi_2-\tilde{\psi}
\\
&\hskip.5cm=\int_{-\infty}^0\omega\cdot(a_1(x+s\omega)-a_1(x+s\omega_1))ds+\int_{-\infty}^0(\omega-\omega_1)\cdot a_1(x+s\omega_1)ds
\\
&\hskip .5cm-\int_{-\infty}^0\omega\cdot(a_2(x+s\omega)-a_2(x+s\omega_2))ds-\int_{-\infty}^0(\omega-\omega_2)\cdot a_2(x+s\omega_2)ds
\end{align*}
and proceeding as above to obtain
\[
\|\Phi_2\|_{0,\infty}\le \tilde{\mathbf{c}}\frac{|\xi|}{\tau}.
\]
This, $\|\Phi_1\|_{0,\infty}\le \tilde{\mathbf{c}}\frac{|\xi|}{\tau}$ together with \eqref{o19} imply
\[
\varphi_1\overline{\varphi_2}=e^{-i\xi\cdot x}e^{i\tilde{\psi}}+\Phi,
\]
where $\Phi$ satisfies
\begin{equation}\label{o20}
\|\Phi\|_{0,\infty}\le \tilde{\mathbf{c}}\frac{|\xi|}{\tau}.
\end{equation}
By combining \eqref{o17} and \eqref{o19}, we find
\begin{align*}
\frac{\mathbf{d}}{\sqrt{\lambda}}\varphi_1\overline{\varphi_2}&=\left(-2a\cdot \omega\overline{\beta}+\mathbf{r}\right)\left(e^{-i\xi\cdot x}e^{i\tilde{\psi}}+\Phi\right)
\\
&:=-2a\cdot \omega\overline{\beta}e^{-i\xi\cdot x}e^{i\tilde{\psi}}+ \mathbf{R},
\end{align*}
where
\[
\mathbf{R}=-2a\cdot \omega\overline{\beta}\Phi+\mathbf{r}\left(e^{-i\xi\cdot x}e^{i\tilde{\psi}}+\Phi\right).
\]
Using \eqref{o16.1}, \eqref{o18}, \eqref{o20} and
\[
\|e^{-i\xi\cdot x}e^{i\tilde{\psi}}+\Phi\|_{0,\infty}\le \tilde{\mathbf{c}},
\]
we get
\[
\|\mathbf{R}\|_{0,1}\le \tilde{\mathbf{c}}\Psi,
\]
where
\[
\Psi:=\frac{|\xi|}{\tau}+\frac{|\xi|^2}{\tau}.
\]
Consequently, we obtain
\begin{equation}\label{o21}
\int_\Omega \frac{\mathbf{d}}{\sqrt{\lambda}}\varphi_1\overline{\varphi_2}=-2\int_\Omega a\cdot \omega\overline{\beta}e^{-i\xi\cdot x}e^{i\tilde{\psi}}dx+\mathbf{I},
\end{equation}
with
\begin{equation}\label{o22}
|\mathbf{I}|\le \tilde{\mathbf{c}}\Psi.
\end{equation}

For $\xi\in \mathbb{R}^n$ such that $\xi_k\ne 0$, $1\le k\le n$, define
\[
\eta_{k,\ell}=\frac{\xi_k\mathbf{e}_\ell +\xi_\ell\mathbf{e}_k}{\sqrt{\xi_k^2+\xi_\ell^2}}.
\]
where $(\mathbf{e}_1,\ldots ,\mathbf{e}_n)$ is the usual Euclidean basis of $\mathbb{R}^n$.

From the calculations in \cite[Lemma 4.1]{Ki}, we obtain from \eqref{o21}
\begin{equation}\label{o23}
\int_\Omega \frac{\mathbf{d}}{\sqrt{\lambda}}\varphi_1\overline{\varphi_2} dx=2\mathscr{F}(\partial_\ell \tilde{a}^k-\partial_k \tilde{a}^\ell)(\xi)+\mathbf{I},\quad \eta=\eta_{k,\ell}.
\end{equation}

We will now apply the identity \eqref{1.13}. To this end, recall that
\begin{align*}
&g_1=i\Delta \psi_1-|\nabla \psi_1|^2-2\nabla\psi_1\cdot a_1+V_{a_1},
\\
&\hat{g}=g_1+2a\cdot(\sqrt{\lambda}\omega_1+\nabla \psi_1)+V_{a_2}-V_{a_1},
\\
&g_2=i\Delta \psi_2-|\nabla \psi_2|^2-2\nabla\psi_2\cdot a_2+V_{a_2},
\\
&h=g_2\alpha+\Delta \alpha +2i(a_2+\nabla \psi_2)\cdot \nabla \alpha.
\\
&f=h +[-2a\cdot(\overline{\sqrt{\lambda}}\omega_2+\nabla\psi_2)+V_{a_1}-V_{a_2}]\alpha+2ia\cdot \nabla \alpha,
\end{align*}
We verify that the following inequalities hold
\begin{align*}
&\|w\|_{0,t}\le \tilde{\mathbf{c}}, \quad w\in \{g_1,g_2\},
\\
&\|\hat{g}\|_{0,t}\le \tilde{\mathbf{c}}\tau,
\\
&\|h\|_{0,t}\le \tilde{\mathbf{c}}(1+|\xi|),
\\
&\|f\|_{0,t}\le \tilde{\mathbf{c}}(\tau+|\xi|).
\end{align*}
In particular, in view of the fact that $\frac{|\xi|}{\tau}\le 1$, we have
\begin{equation}\label{ma1}
\frac{\|g_1\|_{0,t}\|f\|_{0,t}+\|\hat{g}\|_{0,t}\|h\|_{0,t}}{\tau}\le \tilde{\mathbf{c}}.
\end{equation}

By applying \eqref{27} with $\theta=\frac{1}{8}$, using that $p_{\frac{1}{8}}=t$ and
\[
\|\varphi_j\|_{0,\infty}\le  \tilde{\mathbf{c}},\quad j=1,2,
\]
we obtain from \eqref{ma1}
\begin{equation}\label{ma2}
\frac{1}{\tau}\left|\int_\Omega R^{b_1}(\lambda)(g_1\varphi_1)\overline{f\varphi_2}dx-\int_\Omega R^{b_2}(\lambda)(\hat{g}\varphi_1)\overline{h\varphi_2}dx\right|\le \tilde{\mathbf{c}}\tau^{-\frac{3}{4}}.
\end{equation}

In addition, suppose that $|\xi|^8\le \tau$. Then \eqref{o22} yields
\begin{equation}\label{ma3}
|\mathbf{I}|\le \tilde{\mathbf{c}}\tau^{-\frac{3}{4}}.
\end{equation}

Let $\tau_\ast$ be as in Proposition \ref{pro2}. For $\tau \ge \tau_\ast$ and $|\xi|^8\le \tau$, in light of  \eqref{ma2} and \eqref{ma3}, combining \eqref{1.13} and \eqref{o23}, we obtain
\begin{equation}\label{o24}
\tilde{\mathbf{c}}|\mathscr{F}(\partial_\ell \tilde{a}^k-\partial_k\tilde{a}^\ell)(\xi)|\le \tau^{-\frac{3}{4}}+\|\Lambda^1(\lambda)-\Lambda^2(\lambda)\|\|\varphi_1\|.
\end{equation}
Here and in the remainder of this text, $\| \cdot\|$ will denote also the norm of $\mathscr{B}(\mathcal{B}, L^{q}(\Gamma))$.

We verify that
\[
\|\varphi_1\|\le \tilde{\mathbf{c}}\tau^2.
\]
This and \eqref{56} in \eqref{o24} imply
\[
\tilde{\mathbf{c}}|\mathscr{F}(\partial_\ell \tilde{a}^k-\partial_k\tilde{a}^\ell)(\xi)|\le \tau^{-\frac{3}{4}}+\tau^{(\sigma +1)\ell+4}\delta(b_1,b_2).
\]
As the set $\{\xi\in \mathbb{R}^n;\; \xi_k\ne 0,\; 1\le k\le n\}$ is dense in $\mathbb{R}^n$, the inequality above holds for all $\xi \in \mathbb{R}^n$. Hence,
\begin{equation}\label{o25}
\tilde{\mathbf{c}}|\mathscr{F}(\partial_\ell \tilde{a}^k-\partial_k\tilde{a}^\ell)(\xi)|^2\le \tau^{-\frac{3}{2}}+\tau^{2((\sigma +1)\ell+4)}\delta(b_1,b_2)^2.
\end{equation}

Let $0<\gamma \le \frac{1}{8}$ to be specified later. Integrating \eqref{o25} on $B(0,\tau^\gamma)$, we obtain
\[
\tilde{\mathbf{c}}\int_{B(0,\tau^\gamma)}|\mathscr{F}(\partial_\ell \tilde{a}^k-\partial_k\tilde{a}^\ell)(\xi)|^2d\xi\le \tau^{-\frac{3}{2}+n\gamma}+\tau^{2((\sigma +1)\ell+4)+n\gamma}\delta(b_1,b_2)^2.
\]

On the other hand, as $\|\partial_\ell \tilde{a}^k-\partial_k\tilde{a}^\ell\|_{1,2}\le \dot{c}$, where $\dot{c}>0$ is constant only depending on $\tilde{\mathfrak{c}}$ and $\Omega$, we obtain

\[
\int_{|\xi |\ge \tau^\gamma}|\mathscr{F}(\partial_\ell \tilde{a}^k-\partial_k\tilde{a}^\ell)(\xi)|^2d\xi\le \dot{c}\tau^{-2\gamma}.
\]
This inequality and the preceding one, combined with Parseval's identity, gives 
\[
\tilde{\mathbf{c}}\|\partial_\ell \tilde{a}_k-\partial_k \tilde{a}_\ell\|_{0,2}^2\le  \tau^{-\frac{3}{2}+n\gamma}+\tau^{-2\gamma}+\tau^{2((\sigma +1)\ell+4)+n\gamma}\delta(b_1,b_2)^2.
\]
Taking in this inequality $\gamma=\min (\frac{3}{2(n+2)},\frac{1}{8})=2\mathfrak{t}$ yields
\begin{equation}\label{o26}
\tilde{\mathbf{c}}\|\partial_\ell \tilde{a}^k-\partial_k\tilde{a}^\ell\|_{0,2}^2\le  \tau^{-2\mathfrak{t}}+\tau^{2((\sigma +1)\ell+4+n\mathfrak{t})}\delta(b_1,b_2)^2.
\end{equation}
In other words, we have
\[
\tilde{\mathbf{c}}\|\partial_\ell \tilde{a}^k-\partial_k\tilde{a}^\ell\|_{0,2}\le  \tau^{-\mathfrak{t}}+\tau^{(\sigma +1)\ell+4+n\mathfrak{t}}\delta(b_1,b_2).
\]
A standard minimization with respect to $\tau$ enables us to obtain
\[
\|d\tilde{a}\|_{0,2}\le \tilde{\mathbf{c}}\delta (b_1,b_2)^{\beta_1},
\]
where
\[
\beta_1=\frac{\mathfrak{t}}{(\sigma +1)\ell+4+(n+1)\mathfrak{t}}.
\]
The proof of Theorem \ref{thm2} is complete.

\section*{Part II:  Anisotropic case}

\section{Introduction}

\subsection{Notations and definitions}

Let $(M,g)$ be a $C^\infty$-smooth compact Riemannian manifold with boundary $\Gamma$. We use in the rest of this text the Einstein summation convention for subscript quantities. This means that a term appearing twice, as a superscript and a subscript, is assumed to be summed from $1$ to $n$.

If $g=(g_{k\ell})$, then define $|g|=\mathrm{det}(g)$ and let $(g^{k\ell}(\cdot))$ be the inverse matrix of $(g_{k\ell}(\cdot))$.

For all $x\in M$, if there is no confusion, the inner product and the norm associated with the metric $g$ will be denoted respecively by $\langle \cdot|\cdot\rangle $ and $|\cdot|$:
\begin{align*}
&\langle X|Y\rangle=g_{k\ell}X^kY^\ell,\quad X=X^k\frac{\partial}{\partial x_k}\in T_xM,\quad Y=Y^k\frac{\partial}{\partial x_k}\in T_xM,
\\
&|X|^2=\sqrt{\langle X|X\rangle},\quad X=X^k\frac{\partial}{\partial x_k}\in T_xM.
\end{align*}
For $x\in M$, we will use henceforth the usual sharp isomorphism which is given as follows
\[
T_x^\ast M\rightarrow T_xM: A=A_kdx^k\mapsto A^\#=A^k\frac{\partial}{\partial x_k},\quad A^k=g^{k\ell}A_\ell.
\]
We can then define an inner product and a norm on $T_x^\ast M$, denoted again respectively by $\langle \cdot|\cdot\rangle $ and $|\cdot|$, according to the formulas
\begin{align*}
&\langle A|B\rangle =g^{k\ell}A_kB_\ell=g_{k\ell}A^kB^\ell=\langle A^\#|B^\#\rangle,
\\
&\hskip 5cm A=A_kdx^k\in T_x^\ast M,\quad B=B_kdx^k\in T_x^\ast M,
\\
&|A|= \sqrt{\langle A|A\rangle}=\langle A^\#|A^\#\rangle=|A^\#| ,\quad A=A_kdx^k\in T_x^\ast M.
\end{align*}

For convenience, we recall that the Laplace-Beltrami operator is given in local coordinates as follows.
\[
\Delta w:=|g|^{-\frac{1}{2}}\partial_k(|g|^{\frac{1}{2}}g^{k\ell}\partial_\ell w).
\]

In what follows, to $a=(a_1,\ldots ,a_n)\in L^n(M,\mathbb{R}^n)$ we associate the $1$-form $A=a_jdx^j$. The magnetic Laplace-Beltrami operator, corresponding to $a$ or $A$ is given by
\[
\Delta_A:= |g|^{-\frac{1}{2}}(\partial_k+ia_k)(|g|^{\frac{1}{2}}g^{k\ell}(\partial_\ell +ia_\ell)w).
\]
When $a\in W^{1,m}(M,\mathbb{R}^n)$ and $A=a_jdx^j$, the co-differential of $A$ is given by  
\[
\delta (A)=|g|^{-\frac{1}{2}}\partial_k(|g|^{\frac{1}{2}}g^{k\ell}a_\ell).
\]
We verify that 
\[
\Delta_Aw =\Delta +2i\langle A| dw \rangle +i\delta(A)-|A|^2.
\]
For $b=(a,V)\in W^{1,m}(M,\mathbb{R}^n)\times L^m(M,\mathbb{R})$, we use in the following the notations
\begin{align*}
&V_a:= i\delta(A)-|A|^2-V,
\\
&\Delta_bw:= \Delta_Aw-V=\Delta +2i\langle A| dw \rangle+V_aw,
\\
&d_aw:=dw+iwA.
\end{align*}

Let $d\mu=\sqrt{|g|}dx_1\ldots dx_n$ be the Riemannian measure. The norm of $L^r(M):=L^r(M,d\mu)$, $1\le r\le \infty$ will be denoted by $\|\cdot \|_{0,r}$.  We denote the usual scalar product on $L^2(M)$ by $(\cdot ,\cdot)$:
\[
(f,g)=\int_Mf\overline{g}d\mu.
\]
The natural norm on $H^1(M)$ is given by
\[
\|w\|_{1,2}=\left(\|w\|_{0,2}^2+\|dw\|_{0,2}^2\right)^{\frac{1}{2}}.
\]
Here
\[
\|dw\|_{0,2}^2=\int_M\langle dw|d\overline{w}\rangle d\mu.
\]
From Poincarr\'e's inequality,
\[
\varkappa:=\sup\{\|w\|_{0,2};\; w\in H_0^1(M),\; \|dw\|_{0,2}=1\}<\infty,
\]
where $H_0^1(M)$ is the closure of $C_0^\infty (\mathring{M})$ in $H^1(M)$. Also, as $H_0^1(M)$ is continuously embedded in $L^{p'}(M)$, we have
\[
\kappa:=\{\|w\|_{0,p'},\; w\in H_0^1(M);\; \|dw\|_{0,2}=1\}<\infty.
\]
Fix $0<\mathfrak{c}<\kappa^{-1}$ and let
\[
\mathcal{A}:=\{a=(a_1,\ldots ,a_n)\in L^n(M,\mathbb{R}^n);\; \|a\|_{0,n}\le \mathfrak{c}\}.
\]
As for \eqref{1}, we have
\begin{equation}\label{m1}
\mathfrak{c}_-\|dw\|_{0,2}\le \|d_aw\|_{0,2}\le \mathfrak{c}_+\|dw\|_{0,2},
\end{equation}
where $\mathfrak{c}_\pm=1\pm \mathfrak{c}\kappa$.

The other notations we will use are the same as those we use up to now, with $\partial_{\nu_a}$ is given in the present case as follows
\[
\partial_{\nu_a}w=(\partial_{\nu}+iA(\nu))w.
\]

We need to adapt some definitions of Part I. Let $N=(N,g)$ be a $C^\infty$-smooth connected compact manifold with boundary so that $M\subset \mathring{N}$. For $1\le r\le \infty$ and $k\ge 1$ an integer, we set
\[
W^{s,r}_\ast(N,\mathbb{K}^k)=\{f\in W^{s,r}(N,\mathbb{K}^k);\; \mathrm{supp}(f)\subset M\},\quad \mathbb{K}\in \{\mathbb{R},\mathbb{C}\}.
\]
This space is endowed with norm of $W^{s,r}(N,\mathbb{K}^k)$, denoted $\|\cdot\|_{s,r}$, and $W_\ast^{0,p}(N,\mathbb{K}^k)=L^r_\ast(N,\mathbb{K}^k)$.

Fix $V_0\in L_\ast^m(N,\mathbb{R})$ nonnegative and non identically equal to zero. Then define
\[
\mathcal{V}=\{ V\in L_\ast^m(N,\mathbb{R}),\; |V|\le V_0\}.
\]

For $b=(a,V)\in \mathcal{A}\times \mathcal{V}$, the spectral analysis of Part I still applies to the actual operator $-\Delta_b$. The sequence of eigenvalues and the corresponding orthonormal basis of eigenfunctions are denoted respectively by $(\lambda_k^b)$ and $(\phi_k^b)$. Also, set
\[
\psi_k^b:= \partial_{\nu_a}\phi_k^b,\quad k\ge 1.
\]
As in Part I, for $b_j=(a_j,V_j)\in \mathcal{A}\times \mathcal{V}$, $j=1,2$, we set
\[
\delta(b_1,b_2)=\sum_{k\ge 1}k^{-\frac{2\ell}{n}}\left[|\lambda_k^{b_1}-\lambda_k^{b_2}|+\|\psi_k^{b_1}-\psi_k^{b_2}\|_{0,q}\right].
\]

We need also some other definitions similar to those used in Part I. Let
\[
\mathcal{A}_+:=\{a\in W_\ast^{1,\infty}(N,\mathbb{R}^n);\; a_{|M}\in \mathcal{A},\|a\|_{2,\infty}\le \tilde{\mathfrak{c}}\}.
\]
$K\subset \mathring{M}$ compact, and define $\mathscr{B}$ as the set of pairs $(b_1,b_2)=((a_1,V_1),(a_2,V_2))\in [\mathcal{A}_+\times \mathcal{V}]^2$ satisfying $\mathrm{supp}(a_1-a_2)\subset K$ and
\[
|A_1-A_2|\le \mathbf{a},\quad |i\delta(A_1-A_2)-|A_1|^2+|A_2|^2+ V_1-V_2|\le W_0,
\]
where $\mathbf{a}\in L_\ast^\infty (N,\mathbb{R})$ and $W_0\in L^\mathfrak{a}_\ast(N,\mathbb{R})$ are nonnegative and non identically equal to zero. Here $A_j=a_{j,k}dx^k$, $j=1,2$.

With these new definitions, we verify that all the  results of Part I until Proposition \ref{pro2} are still valid with constants depending also on $g$.

\subsection{Main results}

Assume that $M$ and $N$ are simple. By $K\in \{M,N\}$ is simple we mean that $\partial K$ is strictly convex and, for all $x\in K$, $\mathrm{exp}_x: \mathrm{exp}_x^{-1}(K)\mapsto K$ is a diffeomorphism.

It should be noted that the notion of a simple manifold is stable under "small perturbation", meaning that every simple manifold is contained in the interior of another simple manifold.
 
Define
\begin{align*}
&\mathcal{A}_0=\{a\in W_\ast^{2,\infty}(N,\mathbb{R}^n);\; a_{|M}\in \mathcal{A},\; \|a\|_{2,\infty}\le \tilde{\mathfrak{c}}\},
\\
&\mathcal{V}_0=\{(V_1,V_2)\in \mathcal{V}\times \mathcal{V};\; V_1-V_2\in H^2(N),\;  \|V_1-V_2\|_{2,2}\le \tilde{\mathfrak{c}} \}.
\end{align*}

\begin{theorem}\label{thmme1}
$(1)$ Let $b_j=(a,V_j)\in \mathcal{A}_0\times L_\ast^m(N,\mathbb{R})$, $j=1,2$ such that $V_1-V_2\in L_\ast^\mathfrak{a}(N,\mathbb{R})$.  If
\[
\lambda_k^{b_1}=\lambda_k^{b_2},\quad \psi_k^{b_1}=\psi_k^{b_2},\quad k\ge 1,
\]
then $V_1=V_2$.
\\
$(2)$ Let $b_j=(a,V_j)\in \mathcal{A}_0\times \mathcal{V}$, $j=1,2$ such that $(V_1,V_2)\in \mathcal{V}_0$. Then
\[
\|V_1-V_2\|_{0,2}\le \mathbf{c} \delta(b_1,b_2)^{\beta_0},
\]
where 
\[
\beta_0=\frac{1}{4(\sigma+1)\ell+18}.
\]
and $\mathbf{c}=\mathbf{c}(n,N,M,a,g,V_0,\mathfrak{c},\tilde{\mathfrak{c}},\varepsilon,\rho)>0$ is a constant.
\end{theorem}

Let us also recall that for all integer $k\ge 0$ and $A\in H^k(N,T^\ast N)$, there exists a uniquely determined $A^s\in H^k(N,T^\ast N)$ and $F\in H^{k+1}(N)$ such that
\begin{equation}\label{dec1}
A=A^s+dF, \quad \delta A^s=0,\quad F_{|\partial N}=0.
\end{equation}
$A^s$ and $dF$ are called respectively the solenoidal and potential parts of $A$.

\begin{theorem}\label{thmme2}
Let $b_j=(a_j,V_j)\in \mathcal{A}_0\times \mathcal{V}$ and $A_j=a_{j,k}dx^k$, $j=1,2$, be such that $(b_1,b_2)\in \mathscr{B}$. Then
 \[
 \|(A_1-A_2)^s\|_{0,2}\le \mathbf{c}\delta(b_1,b_2)^{\beta_1},
 \]
 where
 \[
 \beta_1=\frac{1}{2(\sigma+1)\ell+10}
 \]
and $\mathbf{c}=\mathbf{c}(n,M,N, g,V_0,\mathfrak{c},\tilde{\mathfrak{c}},\varepsilon,\rho,K)>0$ is a constant.
\end{theorem}

As in the previous section, we give a consequence of Theorems \ref{thmme1} and \ref{thmme2} concerning the uniqueness of the determination of $b=(A,V)$ from the corresponding BSD.

\begin{corollary}
Let $b_j=(a_j,V_j)\in \mathcal{A}_0\times L_\ast^m(N,\mathbb{R})$ and $A_j=a_{j,k}dx^k$, $j=1,2$, be such that $(b_1,b_2)\in \mathscr{B}$. If
\[
\lambda_k^{b_1}=\lambda_k^{b_2},\quad \psi_k^{b_1}=\psi_k^{b_2},\quad k\ge 1,
\]
then $A^s_1=A^s_2$ and $V_1=V_2$.
\end{corollary}

\begin{proof}
First, $A_1^s=A_2^s$ follows from Theorem \ref{thmme2}. We can then proceed as in the last part of the proof of \cite[Theorem 2.2]{BCDKS} to obtain from Theorem \ref{thmme1} that $V_1=V_2$.
\end{proof}

\section{Preliminaries}

In the remainder of this text, $M$ and $N$ are assumed to be simple with $M\subset \mathring{N}$. In the coming two subsections, we recall some results borrowed from \cite{BCDKS}. For further details, we refer the reader to \cite{BCDKS} and the references therein.

\subsection{Geodesic ray transform}

For all $x\in N$, recall that the exponential map $\mathrm{exp}_x$ is defined as folows
\[
\mathrm{exp}_x: v\in T_xN\mapsto \gamma_{x,\theta}(|v|)\in N,\quad \theta =\frac{v}{|v|},
\]
where $\gamma_{x,\theta}$ is the unique geodesic satisfying the initial conditions $\gamma_{x,\theta}(0)=x$ and $\dot{\gamma}_{x,\theta}(0)=v$.

In the following, we use the notations 
\begin{align*}
&SN:=\{ (x,\theta)\in TN;\; |\theta|=1\},
\\
&\partial SN:=\{(x,\theta)\in SN,\; x\in \partial N\},
\\
&\partial_\pm SN:=\{ (x,\theta)\in SN;\; x\in \partial N,\; \pm \langle \theta|\nu(x)\rangle<0\},
\end{align*}
where $\nu$ denotes the unit outer normal vector field on $\partial N$.

As $N$ is simple, for all $(x,\theta)\in SN$, $\mathrm{exp}_x$ is defined in a maximal finite interval $[\ell_-(x,\theta),\ell_+(x,\theta)]$ so that $\ell_\pm(x,\theta)\in \partial N$.

Also, set for $ x\in N$
\begin{align*}
&S_xN:=\{\theta\in T_xN;\; |\theta|=1\},
\\
&S^+_xN:=\{\theta \in S_xN;\; \langle \theta|\nu(x)\rangle >0\},
\end{align*}
and let $d\omega_x(\theta)$ denotes  the Riemannian mesure on $S_xN$, which induces on $SN$ and $\partial SN$ respectively the Riemannian measures  $d\mu(x) d\omega_x(\theta)$ and $ds(x) d\omega_x(\theta)$.

The following Santal\'o's formula will be useful hereinafter
\begin{align*}
&\int_{SN}F(x,\theta)d\mu(x)d\omega_x(\theta)
\\
&\hskip 2cm =\int_{\partial_+SN}\left(\int_0^{\ell_+(x,\theta)}F(\gamma_{x,\theta}(t),\dot{\gamma}_{x,\theta}(t))dt\right) |\langle\theta|\nu(x)\rangle| ds(x) d\omega_x(\theta).
\end{align*}

Recall that the geodesic ray transform for functions is given as follows.
\[
\mathbf{I}_0:C^\infty(N)\rightarrow C^\infty(\partial_+SN): f\mapsto \mathbf{I}_0(f)(x,\theta):=\int_0^{\ell_+(x,\theta)}f(\gamma_{x,\theta}(t))dt.
\]
For all integer $k\ge 0$, $\mathbf{I}_0$ extends to a bounded operator, denoted again by $\mathbf{I}_0$, from $H^k(N)$ to $H^k(\partial_+SN)$. Furthermore,  $\mathbf{I}_0:L^2(N)\mapsto L^2(\partial_+SN)$ is injective.

For the sake of simplicity, $L^2(\partial_+SN,  |\langle\theta|\nu(x)\rangle| ds(x) d\omega_x(\theta))$ will be denoted by $L_w^2(\partial_+SN)$. It has been shown that the adjoint of $\mathbf{I}_0:L^2(N)\mapsto L_w^2(\partial_+SN)$ is given as follows.
\[
\mathbf{I}_0^\ast: L_w^2(\partial_+SN)\rightarrow L^2(N): h\mapsto \mathbf{I}_0^\ast(h)(x)=\int_{S_xN}\tilde{h}(x,\theta)d\omega_x(\theta),
\]
where $\tilde{h}$ is the extension of $h$ from $\partial_+SN$ to $SN$ which is constant on each orbit of the geodesic flow:
\[
\tilde{h}(x,\theta)=h(\gamma_{x,\theta}(\ell_-(x,\theta))).
\]
The operator $\mathbf{N}_0:=\mathbf{I}_0^\ast\mathbf{I}_0: L^2(N)\mapsto L^2(N)$ is usually called the normal operator.

The following double inequality will be useful hereinafter:  there exists two constants $c_0>0$ and $c_1>0$ such that 
 \begin{equation}\label{nor1}
 c_0\|f\|_{0,2}\le \|\mathbf{N}_0(f)\|_{1,2}\le c_1\|f\|_{0,2},\quad f\in L^2(M),
 \end{equation}
 where $L^2(M)$ is viewed as a subspace of $L^2(N)$ via the embedding $f\in L^2(M)\mapsto \chi_Mf\in L^2(N)$, where $\chi_M$ stands for the characteristic function of $M$. Moreover, for all integer $k\ge 0$ and $\mathcal{O}$ an open subset of $N$, there exists a constant $c_{k+1}>0$ such that for all $f\in H^k(N)$ with $\mathrm{supp}(f)\subset \mathcal{O}$ we have
 \begin{equation}\label{nor2}
\|\mathbf{N}_0(f)\|_{k+1,2}\le c_{k+1}\|f\|_{k,2}.
\end{equation}

Next, we review the main properties of the geodesic ray transform for $1$-forms. For $A=a_jdx^j$ with $a=(a_1,\ldots ,a_n)\in C^\infty(N,\mathbb{R}^n)$, define
\[
\mathbf{I}_1(A)(x,\theta):=\int_0^{\ell_+(x,\theta)}a_j(\gamma_{x,\theta}(t))\dot{\gamma}^j_{x,\theta}(t)dt.
\]
That is $\mathbf{I}_1$ acts as an operator from $C^\infty(N,T^\ast N)$ to $C^\infty(\partial_+SN)$. For all integer $k\ge 0$, $\mathbf{I}_1$ has an extension, denoted again by $\mathbf{I}_1$, from $H^k(N,T^\ast N)$ to $H^k(\partial_+SN)$.

Let $A\in H^k(N,T^\ast N)$ and consider its decomposition given by \eqref{dec1}. Then we have $\mathbf{I}_1(dF)=0$ and if $\mathbf{I}_1(A^s)=0$ then $A^s=0$.

$\mathbf{I}^\ast_1:L_w^2(\partial_+SN) \mapsto L^2(N, T^\ast N)$, the adjoint of $\mathbf{I}_1:L^2(N, T^\ast N) \mapsto L_w^2(\partial_+SN)$, is given as follows
\[
[\mathbf{I}_1^\ast(h)(x)]_j=\int_{S_xN} \theta^j\tilde{h}(x,\theta)d\omega_x(\theta).
\]
Here $\tilde{h}$ is the extension of $g$ from $\partial_+SN$ to $SN$ which is constant on each orbit of the geodesic flow, that is
\[
\tilde{h}(x,\theta)=h(\gamma_{x,\theta}(\ell_-(x,\theta)),\dot{\gamma}_{x,\theta}(\ell_-(x,\theta))),\quad (x,\theta)\in SN.
\]
Define the normal operator $\mathbf{N}_1:=\mathbf{I}_1^\ast\mathbf{I}_1:L^2(N, T^\ast N)\rightarrow L^2(N, T^\ast N)$. As for 
$\mathbf{N}_0$, there exists two constants $c_0>0$ and $c_1>0$ such that
\begin{equation}\label{nor4}
c_0\|A^s\|_{0,2}\le \|\mathbf{N}_1(A)\|_{1,2} \le c_1 \|A^s\|_{0,2},\quad f\in L^2(M,T^\ast M),
\end{equation}
where $L^2(M,T^\ast M)$ is considered as a subspace of $L^2(N,T^\ast N)$. Finally, for all integer $k\ge 0$ and $\mathcal{O}$ an open subset of $N$, we find a constant $c_{k+1}>0$ such that for all $A\in H^k(N,T^\ast N)$ satisfying $\mathrm{supp}(A)\subset \mathcal{O}$ we have
\begin{equation}\label{nor5}
\|\mathbf{N}_1(A)\|_{k+1,2}\le c_{k+1}\|A^s\|_{k,2}.
\end{equation}

\subsection{Special solutions}

Fix $\tilde{x}\in \partial N$ arbitrarily and denote $(r,\theta)$ the polar normal coordinates with center $\tilde{x}$. If there is no confusion, for $x=\mathrm{exp}_{\tilde{x}}(r\theta)$, $r>0$ and $\theta\in S_{\tilde{x}}N$, we use indifferently use $w(x)$ or $w(r,\theta)$. With this convention,
\[
g(r,\theta)=dr^2+g_0(r,\theta).
\]

If $d_g$ denotes the geodesic distance, then $\psi(x)=d_g(x,\tilde{x})=r$ belongs to $C^\infty (M)$ and it is the solution of the eikonal equation
\begin{equation}\label{ee}
|d \psi|^2=g^{k\ell}\partial_\ell \psi \partial_k\psi=1.
\end{equation}

Set $\rho(r,\theta)=\mathrm{det}(g_0(r,\theta))$ and let $\eta=\eta(\tilde{x},\cdot) \in H^3(S^+_{\tilde{x}}N)$ and $\alpha =\rho^{-\frac{1}{4}}\eta$. Then $\alpha$ is the solution of the transport equation
\begin{equation}\label{tr1}
2\langle d\alpha|d\psi\rangle+ \Delta \psi \alpha=0.
\end{equation}

Recall that
\[
\mathcal{A}_0=\{a\in W_\ast^{2,\infty}(N,\mathbb{R}^n);\; a_{|M}\in \mathcal{A},\; \|a\|_{2,\infty}\le \tilde{\mathfrak{c}}\}.
\]

Let $A=a_jdx^j$ with $a\in \mathcal{A}_0$, $\sigma(r,\theta)=\langle \dot{\gamma}_{\tilde{x},\theta}(r)|A^\#(\gamma_{\tilde{x},\theta}(r))\rangle$ and 
\[
\tilde{\sigma}(r,\theta)=\int_0^{\ell_+(\tilde{x},\theta)} \sigma(r+t,\theta)dt.
\]
Then $\vartheta=e^{i\tilde{\sigma}}$ is the solution of the following transport equation.
\begin{equation}\label{tr2}
\langle d\vartheta|d\psi\rangle+i \langle A|d\psi\rangle \vartheta =0.
\end{equation}

For $\tau>1$ and $\lambda =(\tau+i)^2$, let 
\[
\varphi_+=\varphi_+(A):=e^{i\sqrt{\lambda}\psi}\alpha\vartheta.
\]

We verify that   
\begin{equation}\label{u1}
(\Delta+\lambda )\varphi_+=\Delta e^{i\sqrt{\lambda}\psi}\alpha\vartheta+2\langle d e^{i\sqrt{\lambda}\psi}|d(\alpha\vartheta)\rangle+e^{i\sqrt{\lambda}\psi}\Delta (\alpha\vartheta)+\lambda \varphi_+.
\end{equation}
Using \eqref{ee}, we obtain
\begin{align*}
\Delta e^{i\sqrt{\lambda}\psi}
&= [-\lambda|d\psi|^2+i\sqrt{\lambda}\Delta \psi]e^{i\sqrt{\lambda}\psi}
\\
&= -\lambda e^{i\sqrt{\lambda}\psi}+i\sqrt{\lambda}\Delta \psi e^{i\sqrt{\lambda}\psi}.
\end{align*}
Hence
\begin{equation}\label{u2}
\Delta e^{i\sqrt{\lambda}\psi}\alpha\vartheta+\lambda \varphi_+= i\sqrt{\lambda}\Delta \psi\varphi_+.
\end{equation}
Using \eqref{u2} in \eqref{u1} gives
\[
(\Delta+\lambda )\varphi_+=i\sqrt{\lambda}\Delta \psi\varphi_++2\langle d e^{i\sqrt{\lambda}\psi}|d (\alpha\vartheta)\rangle+e^{i\sqrt{\lambda}\psi}\Delta(\alpha\vartheta).
\]
This and
\[
2\langle d e^{i\sqrt{\lambda}\psi}|d (\alpha\vartheta)\rangle=2i\sqrt{\lambda}[\langle d \psi|d \alpha\rangle\vartheta+\langle d \psi|d \vartheta\rangle\alpha]e^{i\sqrt{\lambda}\psi}
\]
imply
\[
(\Delta+\lambda )\varphi_+=i\sqrt{\lambda}\Delta \psi\varphi_++2i\sqrt{\lambda}[\langle d \psi|d \alpha\rangle\vartheta+\langle d \psi|d \vartheta\rangle\alpha]e^{i\sqrt{\lambda}\psi} +e^{i\sqrt{\lambda}\psi}\Delta (\alpha \vartheta).
\]

This inequality, combined with \eqref{tr1}, yields
\begin{equation}\label{u3}
(\Delta+\lambda )\varphi_+=2i\sqrt{\lambda}\langle d \psi|d \vartheta\rangle\alpha e^{i\sqrt{\lambda}\psi}+e^{i\sqrt{\lambda}\psi}\Delta (\alpha \vartheta).
\end{equation}

On the other hand, we have
\begin{equation}\label{u4}
2i\langle A|d\varphi_+\rangle =-2\sqrt{\lambda}\langle A|d\psi\rangle \varphi_+ +2i\langle A|d (\alpha\vartheta)\rangle e^{i\sqrt{\lambda}\psi}.
\end{equation}

In view of \eqref{tr2}, putting together \eqref{u3} and \eqref{u4} we obtain
\begin{equation}\label{u5}
(\Delta_b+\lambda)\varphi_+= e^{i\sqrt{\lambda}\psi}f,
\end{equation}
where
\begin{equation}\label{u5.0}
f=\Delta (\alpha \vartheta)+2i\langle A|d (\alpha\vartheta)\rangle +V_a\alpha \vartheta.
\end{equation}

Also, define
\[
\varphi_-=\varphi_-(A):=e^{i\overline{\sqrt{\lambda}}\psi}\alpha\vartheta.
\]
Similarly to \eqref{u5}, we verify that
\begin{equation}\label{u5-}
(\Delta_b+\overline{\lambda})\varphi_-= e^{i\overline{\sqrt{\lambda}}\psi}f.
\end{equation}

\subsection{Integral identity}

We reuse the notations of the preceding subsection. For $j=1,2$, let $b_j=(a_j,V_j)\in \mathcal{A}_0\times\mathcal{V}$, $A_j=(a_j)_kdx^k$, $\varphi_1:= \varphi_+(A_1)$, $\varphi_2=\varphi_-(A_2)$ and $f=f_j$  when $b=b_j$, where $f$ is as in \eqref{u5.0}. Then we have
\begin{align*}
(\Delta_{b_1}-\Delta_{b_2})\varphi_2&=2i \langle A_1-A_2|d\varphi_2\rangle+(V_{a_1}-V_{a_2})\varphi_2
\\
&=-2\overline{\sqrt{\lambda}}\langle A_1-A_2|d\psi\rangle \alpha_2 \vartheta_2 e^{i\overline{\sqrt{\lambda}}\psi}
\\
&\hskip 1.5cm+2i \langle A_1-A_2|d(\alpha_2\vartheta_2)\rangle e^{i\overline{\sqrt{\lambda}}\psi}+(V_{a_1}-V_{a_2})\alpha_2 \vartheta_2 e^{i\overline{\sqrt{\lambda}}\psi}.
\end{align*}
That is we have
\begin{equation}\label{u5.1}
(\Delta_{b_1}-\Delta_{b_2})\varphi_2=g_2e^{i\overline{\sqrt{\lambda}}\psi},
\end{equation}
where
\[
g_2:=-2\overline{\sqrt{\lambda}}\langle A_1-A_2|d\psi\rangle \alpha_2 \vartheta_2 
+2i \langle A_1-A_2|d(\alpha_2\vartheta_2)\rangle +(V_{a_1}-V_{a_2})\alpha_2 \vartheta_2 .
\]
Consequently, we obtain
\begin{equation}\label{u6}
(\Delta_{b_1}+\overline{\lambda})\varphi_2= (f_2+g_2)e^{i\overline{\sqrt{\lambda}}\psi}:=h_2e^{i\overline{\sqrt{\lambda}}\psi}.
\end{equation}

The following Green's formula will be used in the sequel
\begin{equation}\label{mgf}
\int_M(\Delta_b+\lambda)u\overline{v}d\mu=\int_Mu\overline{(\Delta_b+\overline{\lambda})v}d\mu+\int_\Gamma \partial_{\nu_a}u\overline{v}ds-\int_\Gamma u\overline{\partial_{\nu_a}v}ds.
\end{equation}

Let 
\[
u_1:=u^{b_1}(\lambda )(\varphi_1)=\varphi_1+R^{b_1}(\lambda)((\Delta_{b_1}+\lambda)\varphi_1).
\]
In light of \eqref{u5}, we have
\begin{equation}\label{u7}
u_1=\varphi_1+R^{b_1}(\lambda)(e^{i\sqrt{\lambda}\psi}f_1).
\end{equation}

By applying Green's formula \eqref{mgf} with $u=u_1$ and $v=\varphi_2$, we obtain
\[
-\int_\Gamma \Lambda^{b_1}(\lambda)(\varphi_1)\overline{\varphi_2}ds+\int_\Gamma \varphi_1\overline{\partial_{\nu_{a_1}}\varphi_2}ds =\int_Mu_1\overline{(\Delta_{b_1}+\overline{\lambda})\varphi_2}d\mu,
\]
which, combined with \eqref{u6}, yields
\[
-\int_\Gamma \Lambda^{b_1}(\lambda)(\varphi_1)\overline{\varphi_2}ds+\int_\Gamma \varphi_1\overline{\partial_{\nu_{a_1}}\varphi_2}ds =\int_Mu_1e^{-i\sqrt{\lambda}\psi}\overline{h_2}d\mu.
\]
This and \eqref{u7} imply
\begin{align}
&-\int_\Gamma \Lambda^{b_1}(\lambda)(\varphi_1)\overline{\varphi_2}ds+\int_\Gamma \varphi_1\overline{\partial_{\nu_{a_1}}\varphi_2}ds=\int_M\alpha_1\vartheta_1\overline{h_2}d\mu \label{u8}
\\
&\hskip 4cm +\int_Me^{-i\sqrt{\lambda}\psi}R^{b_1}(\lambda)(e^{i\sqrt{\lambda}\psi}f_1)\overline{h_2}d\mu.\nonumber
\end{align}

Similarly to \eqref{u5.1}, we have
\[
(\Delta_{b_2}+\lambda)\varphi_1=(f_1+g_1)e^{i\sqrt{\lambda}\psi}:=h_1e^{i\sqrt{\lambda}\psi},
\]
where
\[
g_1:=-2\sqrt{\lambda}\langle A_2-A_1|d\psi\rangle \alpha_1 \vartheta_1 
+2i \langle A_2-A_1|d(\alpha_1\vartheta_1)\rangle +(V_{a_2}-V_{a_1})\alpha_1 \vartheta_1 .
\]

Next, let 
\begin{align*}
v_1:=u^{b_2}(\lambda)(\varphi_1)&=\varphi_1+R^{b_2}(\lambda)((\Delta_{b_2}+\lambda)\varphi_1)
\\
&= \varphi_1+R^{b_2}(\lambda)(h_1e^{-i\sqrt{\lambda}\psi}).
\end{align*}
Once again, Green's formula with $u=v_1$ and $v=\varphi_2$ yields
\[
-\int_\Gamma \Lambda^{b_2}(\lambda)(\varphi_1)\overline{\varphi_2}ds+\int_\Gamma \varphi_1\overline{\partial_{\nu_{a_2}}\varphi_2}ds =\int_Mv_1\overline{(\Delta_{b_2}+\overline{\lambda})\varphi_2}d\mu.
\]
In light of \eqref{u5}, this inequality gives
\[
-\int_\Gamma \Lambda^{b_2}(\lambda)(\varphi_1)\overline{\varphi_2}ds+\int_\Gamma \varphi_1\overline{\partial_{\nu_{a_2}}\varphi_2}ds =\int_Mv_1\overline{f_2}e^{-i\sqrt{\lambda}\psi}d\mu.
\]
Hence,
\begin{align}
&-\int_\Gamma \Lambda^{b_2}(\lambda)(\varphi_1)\overline{\varphi_2}ds+\int_\Gamma \varphi_1\overline{\partial_{\nu_{a_2}}\varphi_2}ds=\int_M\alpha_1\vartheta_1\overline{f_2}d\mu \label{u9}
\\
&\hskip 4cm +\int_Me^{-i\sqrt{\lambda}\psi}R^{b_2}(\lambda)(e^{i\sqrt{\lambda}\psi}h_1)\overline{f_2}d\mu.\nonumber
\end{align}

Assume that $a_1=a_2$ on $\Gamma$. Whence, $\partial_{\nu_{a_1}}\varphi_2=\partial_{\nu_{a_2}}\varphi_2$. Then taking the difference side by side of \eqref{u8} and \eqref{u9}, we find
\begin{align}
&\int_\Gamma [\Lambda^{b_2}(\lambda)-\Lambda^{b_1}(\lambda)](\varphi_1)\overline{\varphi_2}ds=\int_M\alpha_1\vartheta_1\overline{g_2}d\mu \label{u10}
\\
&\quad+\int_Me^{-i\sqrt{\lambda}\psi}R^{b_1}(\lambda)(e^{i\sqrt{\lambda}\psi}f_1)\overline{h_2}d\mu-\int_Me^{-i\sqrt{\lambda}\psi}R^{b_2}(\lambda)(e^{i\sqrt{\lambda}\psi}h_1)\overline{f_2}d\mu.\nonumber
\end{align}

\section{Proof of the main results}

\subsection{Proof of Theorem \ref{thmme1}}

Let $b_j=(a,V_j)$, $j=1,2$, be as in $(1)$ of Theorem \ref{thmme1} and set $V=V_1-V_2$. With the notations of the preceding section, we choose $\alpha_1=\rho^{-\frac{1}{4}}\eta$ and $\alpha_2=\rho^{-\frac{1}{4}}$. Under these assumptions, we have  $\vartheta_1=\vartheta_2$ and
\[
g_2=(V_1-V_2)\alpha_2\vartheta_2.
\]
Hence, 
\begin{align}
\int_M\alpha_1\vartheta_1\overline{g_2}d\mu&=\int_{S_{\tilde{x}}^+N}\int_0^{\ell_+(\tilde{x},\theta)}V(r,\theta)\eta(\theta)drd\omega_{\tilde{x}}(\theta)\label{u11}
\\
&=\int_{S_{\tilde{x}}^+N}\mathbf{I}_0(V)(\tilde{x},\theta)\eta(\tilde{x},\theta)d\omega_{\tilde{x}}(\theta).\nonumber
\end{align}

Unless otherwise stated, \[\mathbf{c}=\mathbf{c}(n,M,g,a, \max(\|V_1\|_{0,m},\|V_2\|_{0,m},\|V_1-V_2\|_{0,\mathfrak{a}}),\mathfrak{c},\varepsilon,\rho)>0\] will denote a generic constant.
 
Let $r=p_{\frac{1}{4}}=\frac{4n}{2n+1}$ ($p_\theta$ is defined in \eqref{rtp}). As $r<2$, we verify 
\[
\|f_j\|_{0,r}+\|g_j\|_{0,r}\le \mathbf{c}\|\eta(\tilde{x},\cdot)\|_{2,2},\quad j=1,2,
\]
which, combined with \eqref{27}, yields
\begin{align}
&\left|\int_Me^{-i\sqrt{\lambda}\psi}R^{b_1}(\lambda)(e^{i\sqrt{\lambda}\psi}f_1)\overline{h_2}d\mu\right|\label{u12}
\\
&\hskip 2cm +\left|\int_Me^{-i\sqrt{\lambda}\psi}R^{b_2}(\lambda)(e^{i\sqrt{\lambda}\psi}h_1)\overline{f_2}d\mu\right|\le \mathbf{c}\tau^{-\frac{1}{2}}\|\eta(\tilde{x},\cdot)\|_{2,2}.\nonumber
\end{align}

Also, we have
\begin{equation}\label{u13}
\left|\int_\Gamma [\Lambda^{b_2}(\lambda)-\Lambda^{b_1}(\lambda)](\varphi_1)\overline{\varphi_2}ds\right|\le \mathbf{c}\tau^2\|\Lambda^{b_2}(\lambda)-\Lambda^{b_1}(\lambda)\| \|\eta(\tilde{x},\cdot)\|_{3,2},
\end{equation}
where $\|\Lambda^{b_2}(\lambda)-\Lambda^{b_1}(\lambda)\|$ denotes the norm of $\Lambda^{b_2}(\lambda)-\Lambda^{b_1}(\lambda)$ in $\mathscr{B}(\mathcal{B},L^q(\Gamma))$.

Using \eqref{u11}, \eqref{u12} and \eqref{u13} in \eqref{u10}, we obtain

\begin{align*}
&\mathbf{c}\left|\int_{S_{\tilde{x}}^+N}\mathbf{I}_0(V)(\tilde{x},\theta)\eta(\tilde{x},\theta)d\omega_{\tilde{x}}(\theta)\right|
\\
&\hskip 3cm\le \tau^{-\frac{1}{2}}\|\eta(\tilde{x},\cdot)\|_{2,2}+\tau^2\|\Lambda^{b_2}(\lambda)-\Lambda^{b_1}(\lambda)\| \|\eta(\tilde{x},\cdot)\|_{3,2},
\end{align*}
which, combined with Proposition \ref{pro2}, implies
\begin{align}
&\left|\int_{\partial_+SN}\mathbf{I}_0(V)(\tilde{x},\theta)\eta(\tilde{x},\theta)ds(\tilde{x})d\omega_{\tilde{x}}(\theta)\right| \label{u14}
\\
&\hskip 3cm\le \mathbf{c}\left(\tau^{-\frac{1}{2}}+\tau^{(\sigma+1)\ell+4}\delta(b_1,b_2)\right)\|\eta\|_{3,2},\quad \tau\ge \tau^\ast,\nonumber
\end{align}
where $\tau^\ast$ is as in Proposition \ref{pro2}.

As $\delta(b_1,b_2)=0$, \eqref{u14} yields $\mathbf{I}_0(V)=0$. Since $\mathbf{I}_0$ is injective, we conclude that $V=0$. In other words, $V_1=V_2$.

In the remainder of this subsection, $\mathbf{c}=\mathbf{c}(n,M,N,g,a,V_0,\mathfrak{c},\tilde{\mathfrak{c}},\varepsilon,\rho)>0$ will denote a generic constant.

Suppose now that the assumptions of $(2)$ of Theorem \ref{thmme1} hold. In particular, $V:=V_1-V_2\in H^2(N)$, $\mathrm{supp}(V)\subset M$ and $\|V\|_{2,2}\le \tilde{\mathfrak{c}}$.  Whence $\mathbf{N}_0(V)\in H^3(N)$. Therefore, 
\[
\eta(\tilde{x},\theta)=\mathbf{I}_0\mathbf{N}_0(V)(\tilde{x},\theta)|\langle\theta|\nu(\tilde{x})\rangle|
\]
belongs to $H^3(\partial_+SN)$. By choosing that $\eta$ in \eqref{u14}, we obtain 
\begin{align*}
&\left| \int_{\partial_+SN}\mathbf{I}_0(V)(\tilde{x},\theta)\mathbf{I}_0\mathbf{N}_0(V)(\tilde{x},\theta)|\langle\theta|\nu(\tilde{x})\rangle ds(\tilde{x})d\omega_{\tilde{x}}(\theta) \right| 
\\
&\hskip 3cm\le \mathbf{c}\left(\tau^{-\frac{1}{2}}+\tau^{(\sigma+1)\ell+4}\delta(b_1,b_2)\right)\|\mathbf{I}_0\mathbf{N}_0(V)\|_{3,2},\quad \tau\ge \tau^\ast,
\end{align*}
from which we derive
\[
\|N_0(V)\|_{0,2}^2\le  \mathbf{c}\left(\tau^{-\frac{1}{2}}+\tau^{(\sigma+1)\ell+4}\delta(b_1,b_2)\right)\|\mathbf{N}_0(V)\|_{3,2},\quad \tau\ge \tau^\ast.
\]
In light of \eqref{nor1} and \eqref{nor2}, we get from the inequality above
\[
\|V\|_{0,2}^2\le  \mathbf{c}\left(\tau^{-\frac{1}{2}}+\tau^{(\sigma+1)\ell+4}\delta(b_1,b_2)\right)\|V\|_{2,2},\quad \tau\ge \tau^\ast.
\]
Hence,
\begin{equation}\label{u17}
\mathbf{c}\|V\|_{0,2}^2\le  \tau^{-\frac{1}{2}}+\tau^{(\sigma+1)\ell+4}\delta(b_1,b_2),\quad \tau\ge \tau^\ast.
\end{equation}
Minimizing the right hand side of \eqref{u17}, we obtain the following H\"older stability inequality
\[
\|V\|_{0,2}\le \mathbf{c} \delta(b_1,b_2)^{\beta_0},
\]
where 
\[
\beta_0=\frac{1}{4(\sigma+1)\ell+18}.
\]
The proof of Theorem \ref{thmme1} is complete.

\subsection{Proof of Theorem \ref{thmme2}}

Let $b_j=(a_j,V_j)\in \mathcal{A}_0\times \mathcal{V}$ and $A_j=a_{j,k}dx^k$, $j=1,2$, be such that $(b_1,b_2)\in \mathscr{B}$. We choose $\alpha_1$ and $\alpha_2$ as in the preceding subsection. That is $\alpha_1=\rho^{-\frac{1}{4}}\eta$ and $\alpha_2=\rho^{-\frac{1}{4}}$. Since
\[
\sigma_k(r,\theta)=\langle \dot{\gamma}_{\tilde{x},\theta}(r)|A^\#(\gamma_{\tilde{x},\theta}(r)\rangle=[\dot{\gamma}_{\tilde{x},\theta}(r)]^ja_j(\gamma_{\tilde{x},\theta}(r)),\quad k=1,2,
\]
and
\[
\tilde{\sigma}_k(r,\theta)=\int_0^{\ell_+(\tilde{x},\theta)}\sigma_k(r+t,\theta)dt,\quad k=1,2,
\]
we obtain 
\begin{equation}\label{u18}
\tilde{\sigma}_k(0,\theta)=\mathbf{I}_1(A_k)(\tilde{x},\theta),\quad k=1,2.
\end{equation}
Set $\sigma:=\sigma_1-\sigma_2$ and $\tilde{\sigma}:=\tilde{\sigma}_1-\tilde{\sigma}_2$. Then
\begin{align*}
\frac{\partial}{\partial r}\tilde{\sigma}(r,\theta)=\frac{\partial}{\partial r}\int_r^{r+\ell_+(\tilde{x},\theta)}\sigma_k(t,\theta)dt=\tilde{\sigma}_k(r,\theta).
\end{align*}
Using that 
\[
\langle(A_1-A_2)(r,\theta)|d\psi(r,\theta)\rangle\vartheta_1(r,\theta)\overline{\vartheta_2}(r,\theta)=\sigma(r,\theta)e^{i\tilde{\sigma}(r,\theta)},
\]
 we obtain
\begin{align*}
 &\int_0^{\ell_+(\tilde{x},\theta)}\langle( A(r,\theta)|d\psi(r,\theta)\rangle\vartheta_1(r,\theta)\overline{\vartheta_2}(r,\theta)dr=\int_0^{\ell_+(\tilde{x},\theta)}\sigma(r,\theta)e^{i\tilde{\sigma}(r,\theta)}dr
 \\
 &\hskip 2cm=-i\int_0^{\ell_+(\tilde{x},\theta)}\frac{\partial}{\partial r}e^{i\tilde{\sigma}(r,\theta)}dr=i[e^{i\tilde{\sigma}(0,\theta)}-1].
\end{align*}
This and \eqref{u18}  yield
\begin{equation}\label{u19}
\int_0^{\ell_+(\tilde{x},\theta)}\langle A(r,\theta)|d\psi(r,\theta)\rangle\vartheta_1(r,\theta)\overline{\vartheta_2}(r,\theta)dr=i\left[e^{i\mathbf{I}_1(A)(\tilde{x},\theta)}-1\right].
\end{equation}
We split $\frac{g_2}{\overline{\sqrt{\lambda}}}$ into two terms
\begin{equation}\label{u20}
\frac{g_2}{\overline{\sqrt{\lambda}}}=-2\langle A|d\psi\rangle\alpha_1\vartheta_2+\tilde{g}_2,
\end{equation}
where
\[
\tilde{g}_2=\frac{2i}{\overline{\sqrt{\lambda}}}\langle A|d(\alpha_2\vartheta)\rangle+\frac{1}{\overline{\sqrt{\lambda}}}(V_{a_1}-V_{a_2})\alpha_2\vartheta_2.
\]
For further use, we note that we have
\begin{equation}\label{u21}
\left|\int_M\overline{\tilde{g}_2}\alpha_1\vartheta_1 d\mu\right|\le \mathbf{c}\tau^{-1}\|\eta\|_{1,2}.
\end{equation}
Here and until the end of the proof, $\mathbf{c}=\mathbf{c}(n,M,N, g,V_0,\mathfrak{c},\tilde{\mathfrak{c}},\varepsilon,\rho,K)>0$ denotes a generic constant.

Next, we combine the following equality
\begin{align*}
&\int_M\frac{g_2}{\overline{\sqrt{\lambda}}}\alpha_1\vartheta_1 d\mu=\int_M\overline{\tilde{g}_2}\alpha_1\vartheta_1 d\mu
\\
&\hskip 2cm -2\int_{\partial_+SN}\langle A(r,\theta)|d\psi(r,\theta)\rangle\vartheta_1(r,\theta)\overline{\vartheta_2}(r,\theta)\eta(\tilde{x},\theta)ds(\tilde{x})d\omega_x(\theta)
\end{align*}
and \eqref{u19} to obtain

\begin{align}
&\int_M\frac{g_2}{\overline{\sqrt{\lambda}}}\alpha_1\vartheta_1d\mu=\int_M\overline{\tilde{g}_2}\alpha_1\vartheta_1d\mu\label{u22}
\\
&\hskip 2cm -2i\int_{\partial_+SN}\left[e^{i\mathbf{I}_1(A)(\tilde{x},\theta)}-1\right]\eta(\tilde{x},\theta)ds(\tilde{x})d\omega_x(\theta).\nonumber
\end{align}

Putting together \eqref{u21} and \eqref{u22}, we find
\begin{align}
&\mathbf{c}\left|\int_{\partial_+SN}\left[e^{i\mathbf{I}_1(A)(\tilde{x},\theta)}-1\right]\eta(\tilde{x},\theta)ds(\tilde{x})d\omega_x(\theta)\right|\label{u23}
\\
&\hskip 3cm\le \tau^{-1}\|\eta\|_{1,2}+\left| \int_M\frac{g_2}{\sqrt{\lambda}}\alpha_1\vartheta_1 d\mu\right|.\nonumber
\end{align}
In light of \eqref{u10}, \eqref{u12} and \eqref{u13}, we get from \eqref{u23}
\begin{align}
&\mathbf{c}\left|\int_{\partial_+SN}\left[e^{i\mathbf{I}_1(A)(\tilde{x},\theta)}-1\right]\eta(\tilde{x},\theta)ds(\tilde{x})d\omega_x(\theta)\right|\label{u24}
\\
&\hskip 3cm\le ( \tau^{-1}+ \tau^2\|\Lambda^{b_1}(\lambda)-\Lambda^{b_2}(\lambda)\|)\|\eta\|_{3,2}.\nonumber
\end{align}
Here and henceforth, $\|\Lambda^{b_1}(\lambda)-\Lambda^{b_2}(\lambda)\|$ denotes again the norm of $\Lambda^{b_1}(\lambda)-\Lambda^{b_2}(\lambda)$ in $\mathscr{B}(\mathcal{B},L^q(\Gamma))$.

We argue as in the proof of \cite[Lemma 4.2]{LQSY} to deduce from \eqref{u24}  
\begin{align}
&\mathbf{c}\left| \int_{\partial_+SN} \mathbf{I}_1(A)(\tilde{x},\theta)\eta(\tilde{x},\theta)ds(\tilde{x})d\omega_x(\theta) \right| \label{u25}
\\
&\hskip 4.5cm \le ( \tau^{-1}+ \tau^2\|\Lambda^{b_1}(\lambda)-\Lambda^{b_2}(\lambda)\|)\|\eta\|_{3,2}.\nonumber
\end{align}
By combining \eqref{56} and \eqref{u25}, we get
\begin{align}
&\mathbf{c}\left| \int_{\partial_+SN} \mathbf{I}_1(A)(\tilde{x},\theta)\eta(\tilde{x},\theta)ds(\tilde{x})d\omega_x(\theta) \right| \label{u26}
\\
&\hskip 4cm \le ( \tau^{-1}+ \tau^{(\sigma+1)\ell +4}\delta(b_1,b_2))\|\eta\|_{3,2}.\nonumber
\end{align}

As $W^{2,\infty}(N,T^\ast N)$ is continuously embedded in $H^2(N,T^\ast N)$ (e.g. \cite[subsection 1.4.4]{Gr}), we have $A\in H^2(N,T^\ast N)$. In view of \eqref{nor4} and \eqref{nor5}, we proceed as in the preceding section to obtain the following stability inequality
 \[
 \|A^s\|_{0,2}\le \mathbf{c}\delta(b_1,b_2)^{\beta_1},
 \]
 where
 \[
 \beta_1=\frac{1}{2(\sigma+1)\ell+10}.
 \]
This completes the proof of Theorem \ref{thmme2}.

\bigskip
\noindent
\textbf{Conflict of interest.}
The authors declare that they have no conflict of interest.

\smallskip
\noindent
\textbf{Data availability.}
Data sharing not applicable to this article as no datasets were generated or analyzed during the current study.

\end{document}